\documentclass[11pt,reqno]{amsart}

\usepackage[T1]{fontenc}
\usepackage[utf8]{inputenc}

\usepackage{tgheros}

\usepackage{amsmath,amssymb,amsfonts,amsthm}
\usepackage{mathtools}
\usepackage{mathrsfs}
\usepackage{bm}
\usepackage{stmaryrd}
\usepackage{bbm}

\usepackage{float}
\usepackage{placeins}

\usepackage{enumitem}
\usepackage{comment}
\usepackage{verbatim}
\usepackage{etoolbox}

\usepackage[final,nopatch=footnote]{microtype}
\usepackage[margin=1in]{geometry}

\usepackage[dvipsnames]{xcolor}
\usepackage[colorlinks=true,linkcolor=blue!80!black,citecolor=red!70!black,urlcolor=MidnightBlue]{hyperref}

\usepackage{tikz}

\theoremstyle{plain}
\newtheorem{thm}{Theorem}[section]
\newtheorem{lem}[thm]{Lemma}

\theoremstyle{definition}

\theoremstyle{remark}
\newtheorem{rem}[thm]{Remark}

\numberwithin{equation}{section}
\newcommand{\bw}{\mathbf{w}}

\newcommand{\ts}{\mathsf{s}}

\newcommand{\cd}{\mathscr{D}}

\newcounter{const}
\definecolor{lime}{HTML}{A6CE39}
\DeclareRobustCommand{\orcidicon}{%
  \begin{tikzpicture}[x=1em,y=1em,baseline=-0.05em]
    \draw[lime, fill=lime] (0,0) circle [radius=0.16]
      node[white] { \tiny ID};
    \draw[white, fill=white] (-0.0625,0.095) circle [radius=0.007];
  \end{tikzpicture}\hspace{-2mm}%
}

\newcommand{\orcidauthorB}{0000-0001-6038-1069}
\newcommand{\orcidB}{\href{https://orcid.org/\orcidauthorB}{\orcidicon}}

\title[Fujita Phenomenon]{\textsc{On the Fujita Phenomenon for a Forced Spatio-Temporal Fractional Diffusion Equation}}

\author[ R. Ben Belgacem and M. Majdoub]{Rihab Ben Belgacem and Mohamed Majdoub\orcidB{}}

\address[R. Ben Belgacem]{Universit\'e de Tunis El Manar, Facult\'e des Sciences de Tunis, D\'epartement de math\'ematiques, Laboratoire \'equations aux d\'eriv\'ees partielles (LR03ES04), 2092 Tunis, Tunisie.}
\email{\sl {\textcolor{blue}{rihabbenbelgacem662@gmail.com  }}}
\address[M. Majdoub]{Department of Mathematics, College of Science, Imam Abdulrahman Bin Faisal University, P. O. Box 1982, Dammam, Saudi Arabia.}
\address[M. Majdoub]{Basic and Applied Scientific Research Center, Imam Abdulrahman Bin Faisal University, P.O. Box 1982, 31441, Dammam, Saudi Arabia.}
\email{\sl {\textcolor{blue}{mmajdoub@iau.edu.sa}}}
\email{\sl {\textcolor{blue}{med.majdoub@gmail.com}}}
\email{\sl {\textcolor{blue}{mohamed.majdoub@fst.rnu.tn}}}

\begin{document}
\begin{abstract}
We investigate the Cauchy problem for a semilinear spatio--temporal fractional diffusion equation with a time-dependent forcing term:
\[
\partial_t^\alpha u + (-\Delta)^{\ts} u = |u|^p + t^{\sigma}\,\bw(x),
\quad (t,x) \in (0,\infty) \times \mathbb{R}^N,
\]
where $\alpha,\ts \in (0,1)$, $\sigma > -\alpha$, and $\bw$ is a given continuous function.
Here $\partial_t^\alpha$ denotes the Caputo fractional derivative.
Our main results are threefold.
First, we establish local-in-time existence of mild solutions and prove finite-time blow-up in the subcritical regime, under the positivity condition
\[
\int\limits_{\mathbb{R}^N} \bw(x)\,dx > 0.
\]
Second, in the supercritical case $-\alpha < \sigma < 0$, we prove the global existence of solutions for sufficiently small initial data and forcing term, and we identify the corresponding critical exponent as
\[
p_F=\frac{N\alpha-2\ts\sigma}{N\alpha-2\ts(\alpha+\sigma)}.
\]
Finally, within this supercritical range, we obtain a more robust global existence result under weaker assumptions that require only local smallness and controlled growth of the data. To the best of our knowledge, a sharp Fujita-type threshold for fully spatio-temporal fractional diffusion equations with time-growing external forcing has not been previously established.
\end{abstract}

\date{\today}

	\maketitle
\setcounter{tocdepth}{1}

%\tableofcontents

%%%%%%%%%%%%%%%%%%%%%%%%%%%%%%%%%%%%%%%%%%%%%%%%%%%%%%%%%%%%%%%%%%%%%%%%%%%%%%%%%%%%%%%%%%%%%%%%%%%%%%%%%%%%%%%%%%%%%%%%%%%%

\section{Introduction and main results}
\label{sec:intro}
In this paper, we investigate the Cauchy problem for a spatio-temporal fractional diffusion equation subject to an external forcing term:
\begin{equation}
\begin{cases}
\partial_t^\alpha u + (-\Delta)^\ts u = |u|^p + t^\sigma \bw(x) & \text{in } (0, \infty) \times \mathbb{R}^N, \\
u(0, x) = u_0(x) \geq 0 & \text{in } \mathbb{R}^N,
\end{cases}
\label{main}
\end{equation}
where $\alpha, \ts \in (0,1)$, $\sigma > -1$, $p > 1$, and $\bw$ is a given continuous function. In what follows, $\partial_t^\alpha$ for $\alpha \in (0,1)$ is the Caputo fractional derivative,
\[
    \partial_t^\alpha u(t,x) = \frac{1}{\Gamma(1-\alpha)} \int_0^t (t-\tau)^{-\alpha} \partial_\tau u(\tau,x) \, d\tau,
\]
and $(-\Delta)^\ts$ for $\ts \in (0,1)$ is the fractional Laplacian, defined spectrally as $(-\Delta)^\ts u := \mathcal{F}^{-1}\big( |\xi|^{2\ts} \mathcal{F}(u) \big)$ with $\mathcal{F}$ denoting the Fourier transform \cite{Book-FC, M, Guide, Vald}.

The framework of fully nonlocal equations, which combines fractional time derivatives with fractional spatial operators, offers a powerful approach for capturing phenomena characterized by long-range interactions and memory effects. Such models naturally occur in a variety of physical settings, including plasma physics and turbulent transport where anomalous diffusion arises from long-range correlations in space and time \cite{Castillo2004, Castillo2005} geophysical flows and atmospheric dynamics, where memory and large scale interactions play a fundamental role \cite{Metzler2000}, as well as in financial mathematics for asset price dynamics incorporating long-term memory and jump processes \cite{Cartea2007}.

From an analytical point of view, identifying the critical exponent that separates global-in-time existence from finite-time blow-up is a central theme. The classic starting point is Fujita's seminal work \cite{fujita, Fuj69} on the semi-linear heat equation, which corresponds in our framework to the case $\alpha = 1$, $\ts = 1$ and $\bw\equiv 0$. He established that for $1 < p < 1 + \frac{2}{N}$, any solution resulting from a non-trivial initial condition $u_0 \not\equiv 0$ blows up in finite time, whereas for $p > 1 + \frac{2}{N}$, global solutions exist if the initial condition is sufficiently small. The critical value $p_F = 1 + \frac{2}{N}$ was subsequently incorporated into the blow-up regime by Hayakawa \cite{Hay, KST77}. These fundamental results are detailed in Levine's summary \cite{DL}.

Subsequently, the research then focused on specific settings. For $\alpha=1$ and $\sigma=0$, the problem was analyzed in \cite{Beri-arXiv}, while a more general framework with $\alpha=1$ was later developed in \cite{Rihab2}. Related nonlocal configurations involving different fractional diffusion operators were also studied in \cite{Rihab1}.

In the classical setting corresponding to $\alpha=\ts=1$, there are many related works in the literature, notably \cite{Opuscula, BLZ, JKS, MM}. In particular, for $\sigma\in(-1,0)$, the critical exponent admits the explicit expression

$$p_F=\frac{N-2\sigma}{\,N-2\sigma-2\,}, $$

which coincides with the result \cite[Theorem 1.1, (1.8)]{JKS}. The case $\alpha=1$ with an integer spatial order $\ts\geq1$ has also been studied in detail in \cite{Majd}.

The Fujita phenomenon was later generalized beyond the classical framework. For Problem~\eqref{main} without a forcing term, with $\ts = 1$ and $0 < \alpha < 1$, Zhang and Sun \cite{ZhangSun} determined the analogous critical exponent $p_F = 1 + \frac{2}{N}$. For the fully fractional situation, Kirane et al.~\cite{KLT} analyzed the more general equation
\begin{equation}
\partial_t^\alpha u + (-\Delta)^{\ts} u = h(x,t)\,|u|^{p}, 
\label{STFE}
\end{equation}
under the assumption $h(x,t) \geq C_h |x|^\sigma t^\rho$. They showed that no global solutions exist whenever
\[
1 < p \leq 1 + \frac{\alpha(2\ts + \sigma) + 2\ts \rho}{\alpha N + 2\ts(1 - \alpha)},
\]
thereby extending the Fujita-type blow-up mechanism to equations involving both time-fractional and space-fractional operators.

This direction of research fits within a broad program aimed at characterizing critical exponents for nonlinear evolution equations. In particular, Guedda and Kirane \cite{Kirane1, Kirane2} developed a unified methodology for proving nonexistence of global solutions across a wide class of problems. Their framework includes the case $\alpha = 1$ in \eqref{STFE}, and clarifies how the interplay between the coefficient $h$, the parameters $p$, $\ts$, and the spatial dimension $N$ determines the critical threshold. Their approach, relying on carefully chosen scaled test functions, yields a powerful and flexible tool for identifying critical exponents influenced by both the diffusion operator and the growth of the nonlinearity.

For the unforced version of our main problem \eqref{main}, the critical exponent can be inferred from the scaling properties of the equation. If $u(t,x)$ is a solution, then for any $\lambda > 0$ the rescaled function
\[
u_\lambda(t,x) := \lambda^{\frac{2\ts}{p-1}} 
u\!\left(\lambda^{\frac{2\ts}{\alpha}} t,\, \lambda x\right)
\]
is again a solution.  
The corresponding scale-invariant Lebesgue space for the initial data is $L^{p_c}(\mathbb{R}^N)$, where
\begin{equation}
\label{p-c}
p_c = \frac{N(p-1)}{2\ts}.
\end{equation}
This leads to the Fujita-type critical exponent
\[
p_F := 1 + \frac{2\ts}{N},
\]
which reduces to the classical value when $\alpha = \ts = 1$.

It is well known that critical exponents are highly sensitive to the underlying structure of the equation. In particular, the presence of additional mechanisms may significantly modify the threshold separating global-in-time existence from finite-time blow-up. For the classical heat equation, for instance, it has been shown that forcing terms can shift the Fujita exponent \cite{BLZ}. Likewise, geometric effects may alter the critical value, as demonstrated by Bandle and Levine \cite{BL89} in the case of conical domains.

A central contribution of the present work is the precise characterization of the impact of an external forcing term on the qualitative behavior of solutions. Our analysis relies on the general criticality framework developed in \cite{KLT}. In this setting, we investigate how the critical blow-up exponent associated with the unforced problem,
\[
p_F = 1 + \frac{2\ts\alpha}{N\alpha + 2\ts(1 - \alpha)},
\]
is modified by the inclusion of an external source.

Our main results identify a sharp critical exponent for the forced equation, which accurately delineates the transition between global-in-time existence and finite-time blow-up. In particular, they show that the presence of a forcing term accelerates the blow-up mechanism when compared to the unforced case, in the sense that the range of exponents $p$ for which solutions necessarily blow up in finite time becomes strictly larger. More precisely, space- and time-dependent forcing effects tend to counterbalance the regularizing influence of fractional diffusion, as governed by the operators $\partial_t^\alpha$ and $(-\Delta)^\ts$.

From a modeling perspective, the incorporation of a time-dependent forcing term of the form $t^{\sigma} \bw(x)$ is well motivated in physical systems exhibiting anomalous diffusion and long-memory effects. Time-fractional diffusion equations involving the Caputo derivative are widely employed to describe subdiffusive transport in disordered media, plasma physics, and turbulent transport, where nonlocal temporal effects capture the persistent influence of past states \cite{Castillo2004, Castillo2005, Diethelm, KST}. In such contexts, external sources typically act continuously in time rather than instantaneously, giving rise to forcing mechanisms whose intensity evolves according to a power law. Closely related phenomena also occur in geophysical flows and atmospheric dynamics, where sustained energy injection interacts with nonlocal diffusion processes \cite{Metzler2000}. The exponent $\sigma$ thus encodes the temporal balance between external forcing and fractional dissipation, and its interplay with the nonlinear reaction term can have a decisive impact on the qualitative behavior of solutions. 
To rigorously quantify the effect of external forcing on the classical Fujita threshold and on the global dynamics of the system, it is necessary to first establish a local well-posedness theory. To this end, we prove the existence and uniqueness of mild solutions in suitable Lebesgue spaces, which provides the analytical foundation for the subsequent study of finite-time blow-up and global-in-time existence. Accordingly, we begin our analysis by deriving a local well-posedness result in the Lebesgue space framework under appropriate assumptions on the data.

\begin{thm}
\label{thm:Loc-Exis}
Let $p>1$, $\alpha, \ts \in (0,1)$, $\sigma > -\alpha$, and $q \in (p_c, \infty]$ with $q \geq p$, where the critical exponent $p_c$ is given by \eqref{p-c}. Assume that the initial data $u_0$ and the forcing term $\bw$ both belong to $L^q(\mathbb{R}^N)$.  
Then, there exists a time $T>0$ such that Problem \eqref{main} admits a unique solution 
\[
u \in L^\infty(0, T; L^q(\mathbb{R}^N)).
\]
\end{thm}
 
\begin{rem}
~\rm
\begin{enumerate}[label=(\roman*)]

\item 
Local existence may also be obtained via semigroup methods.  
More precisely, one can introduce the solution operators  
\( \mathbb{P}_{\alpha,\ts}(t) \) and \( \mathbb{S}_{\alpha,\ts}(t) \) through the subordination principle, following \cite{Bazhlekova2000}:
\begin{align*}
\mathbb{P}_{\alpha,\ts}(t) u_{0}
    &= \int_{0}^{\infty} \phi_{\alpha}(\theta)\, 
       T_{\ts}(t^{\alpha}\theta)\, u_{0}\, d\theta,\\
\mathbb{S}_{\alpha,\ts}(t) \varphi
    &= \alpha \int_{0}^{\infty} \theta\, \phi_{\alpha}(\theta)\, 
       T_{\ts}(t^{\alpha}\theta)\, \varphi\, d\theta,
\end{align*}
where \( T_{\ts}(t) = e^{-t(-\Delta)^{\ts}} \) denotes the fractional heat semigroup and  
\( \phi_{\alpha} \) is the Wright-type function (see \cite{Mainardi1996}).  
This approach yields mild solutions in \( C([0,T], L^{q}(\mathbb{R}^{N})) \).  
In the present work, however, we rely on the Duhamel representation involving the kernels \( Z \) and \( Y \), which is better adapted to capturing spatial decay and regularity effects.

\item 
In the case \( \ts = 1 \) and \( \bw = 0 \), Zhang and Sun \cite{ZhangSun} established local well-posedness for initial data in  
\( C_{0}(\mathbb{R}^{N}) \) and provided a complete characterization of the critical exponent governing the transition between finite-time blow-up and global existence.

\item 
The classical situation \( \alpha = \ts = 1 \) was previously studied in \cite{JKS} under the assumptions  
\( \sigma > -1 \) and \( \sigma \neq 0 \).  
In that setting, both local well-posedness and the corresponding critical exponent were rigorously derived.

\end{enumerate}
\end{rem}

Our next result addresses the nonexistence of global solutions to~\eqref{main}. The precise statement is as follows:
\begin{thm}\label{thm:blowup}  
Let $N \geqslant 1$ and $\alpha,\ts \in (0,1)$. 
Assume further that ${\mathbf w} \in C_0(\mathbb{R}^N) \cap L^1(\mathbb{R}^N)$ satisfies
\begin{equation}
\label{Ass-Forced}
    \int\limits_{\mathbb{R}^N} {\mathbf w}(x)\, dx > 0.
\end{equation}
\begin{enumerate}[label=(\roman*)]
    \item If $-\alpha<\sigma \leq 0$, $2\ts < N$, and 
    \begin{equation}
    \label{Fuji-forced}
        1 < p < p^* \vcentcolon=\frac{N\alpha-2\ts\sigma}{N\alpha-2\ts(\alpha+\sigma)},
    \end{equation}
    then problem~\eqref{main} admits no global weak solution.
    \item If $\sigma > 0$, then the same conclusion holds for every $p > 1$.
\end{enumerate}
\end{thm}

\begin{rem}
~\rm
\begin{enumerate}[label=(\roman*)]
    \item The condition $N > 2\ts$ ensures that the exponent $p^*$ defined in~\eqref{Fuji-forced} satisfies $p^* > 1$.

    \item A crucial factor for blow-up is the positivity of the total mass of $\bw$. In particular, when $\sigma > 0$, the temporal amplification enhances the effect of the forcing, thereby lowering the critical exponent $p^*$ beyond which solutions necessarily blow up.  
    \item These results emphasize the delicate balance between dissipation due to fractional diffusion and amplification induced by external forcing, highlighting the significant role of the temporal factor $t^\sigma$ in governing blow-up.
    \end{enumerate}
    \end{rem}

We now turn to the analysis of the global theory. Global existence is established under the assumption that the initial data is sufficiently small in the critical norm, and that the forcing term is suitably small in an appropriate functional norm.  The main result in this setting is stated as follows:
\begin{thm}
\label{thm:glob-exis-1}
Let $N\geq1$, $\alpha, \ts \in (0,1)$, and $-\alpha<\sigma \leq 0$. Assume that $p \geq \frac{N\alpha-2\ts\sigma}{N\alpha-2\ts(\alpha+\sigma)}$, where $ N>2\ts$. 
Let us define ${r_c}=\frac{N\alpha(p-1)}{2\ts\sigma(p-1)+2\ts\alpha p}$ and recall that $p_c$ is given by \eqref{p-c}. Then, for any initial data  $u_0\in L^{p_c}(\mathbb{R}^N)$ and $\bw \in L^{r_c}(\mathbb{R}^N)$, with the property that $\|u_0\|_{L^{p_c}} + \|\bw\|_{L^{r_c}}$ is sufficiently small, the equation \eqref{main} admits a global solution $u$ in time. 
\end{thm}

\begin{rem}
\rm 
~\begin{enumerate}[label=(\roman*)]
    \item When $\sigma = 0$, the scaling properties imply that $L^{p_c}(\mathbb{R}^N)$ is the unique Lebesgue space with the following feature: if the initial data are small in this space and the forcing term $\bw$ is small in $L^{\frac{p_c}{p}}(\mathbb{R}^N)$, then a global solution exists in $L^q$.  

    \item In the absence of a forcing term (i.e., $\bw = 0$), the problem is addressed in \cite{CVPDE-2024}. There, global existence is established either under a strong condition, requiring the smallness of $\|u_0\|_{L^{p_c}(\mathbb{R}^N)}$, or under a weaker condition, relying on smallness of local averages of $u_0$.

    \item The global existence result can be extended to $L^q$ spaces outside the distinguished range, provided additional integrability conditions are imposed on $u_0$. In the critical case $p^* = 1 + \frac{2\ts}{N}$, the existence of global solutions depends on $\alpha$: for $\alpha \in (0,1)$, global solutions exist, whereas for $\alpha = 1$, every nonnegative nontrivial solution blows up in finite time.
\end{enumerate}
\end{rem}

When a forcing term depends on both time and space, a weaker criterion for the global existence of $L^q$-solutions can be formulated by imposing smallness conditions on the local averages of the initial data $u_0$ and the forcing term $\bw$ over balls of suitable radius. The following result provides a global existence criterion based on local averages rather than norm smallness.
\begin{thm}
\label{thm:glob-exis-2} 
Let $\alpha, \ts \in (0,1)$, $-\alpha < \sigma < 0$, and $2\ts < N$. Assume 
$$
\frac{N\alpha - 2\ts \sigma}{N\alpha - 2\ts(\alpha + \sigma)}\leq p  <\frac{N\alpha-2\ts q\sigma}{N\alpha-2\ts q(\alpha+\sigma)}$$
and let $q$ satisfy 
$$
p_c < q \le p p_c, \quad q \ge p,
$$
where $p_c$ is given by \eqref{p-c}.
Let $u_0, \bw \in L^{p_c}(\mathbb{R}^N) \cap L^q(\mathbb{R}^N)$, with $u_0,\bw \geq 0$. Assume the scaling condition 
$$-(\sigma + \alpha)<0 \leq \beta := \frac{\alpha}{p-1} - \frac{N\alpha}{2\ts q}<\frac{1}{p}.$$
Let $R_0 = \min\left\{ 
\left( \frac{\mathscr{M}}{2\|u_0\|_{L^q}} \right)^{\frac{1}{\beta}}, 
\left( \frac{\mathscr{M}}{2C\|\bw\|_{L^q}} \right)^{\frac{1}{\beta+\alpha+\sigma}} 
\right\}$ with $\mathscr{M}$ from Lemma~\ref{lem:W-Condition}, and choose $M \geq 0$, $\delta > 0$ such that
\begin{align}
& \|u_0\|_{L^{p_c}(|x| > M R_0^{\alpha/2\ts})} \leq \frac{\mathscr{M}}{4\|F\|_{L^r}}, \label{H1} \\
& \|w\|_{L^{p_c}(|x| > M R_0^{\alpha/2\ts})} \leq \frac{\mathscr{M}}{4\|F\|_{L^r}}, \label{H2} \\
& \|F\|_{L^r(|\xi| < \delta)} \leq \frac{\mathscr{M}}{\max\left(\|u_0\|_{L^{p_c}}, \|\bw\|_{L^{p_c}}\right)}, \label{H3}
\end{align}
with $1 - \dfrac{1}{r} = \dfrac{1}{p_c} - \dfrac{1}{q}$, If the local averages satisfy
\begin{align}
& \sup_{R > M R_0} R^{\frac{\alpha}{p-1} - \frac{N\alpha}{2\ts}} \int\limits_{|y| < R^{\frac{\alpha}{2\ts}}} u_0(y) \, dy < \frac{\mathscr{M} M^{\frac{2\ts}{p-1} - N}}{8\|F\|_{L^q(|\xi| > \delta)}}, \label{A1} \\
& \sup_{R > M R_0} R^{\frac{\alpha}{p-1} - \frac{N\alpha}{2\ts} + \sigma + \alpha} \int\limits_{|y| < R^{\frac{\alpha}{2\ts}}} \bw(y) \, dy < \frac{\mathscr{M} M^{\frac{2\ts}{p-1} - N}}{8\|F\|_{L^q(|\xi| > \delta)}}, \label{A2}
\end{align}
then the maximal $L^q$-solution exists globally in time.
\end{thm} 

 The two global existence results, Theorem~\ref{thm:glob-exis-1} and Theorem~\ref{thm:glob-exis-2}, are established under distinct sets of assumptions. In Theorem~\ref{thm:glob-exis-1}, global existence is ensured by imposing smallness conditions on the norms of the initial data $u_0$ and the forcing term $\bw$. Specifically, these smallness conditions are formulated in terms of norms associated with the critical exponent $p_c$ and the exponent $r_c$ introduced in the proof.  

By contrast, Theorem~\ref{thm:glob-exis-2} does not require smallness in any norm. Instead, it relies on assumptions concerning the behavior and decay of local averages of $u_0$ and $\bw$ over balls. This approach emphasizes the spatial distribution and asymptotic properties of the data rather than their global magnitude.  

Therefore, while Theorem~\ref{thm:glob-exis-1} is grounded in quantitative smallness within specific functional spaces, Theorem~\ref{thm:glob-exis-2} provides an alternative existence criterion based on qualitative spatial decay and distribution properties.\\

\noindent{\bf Outline of the paper.}  
The paper is organized as follows.  

In Section~\ref{sec:prelim}, we collect the preliminary material needed throughout the paper. 
We begin by reviewing essential concepts related to fractional derivatives and then introduce the fundamental kernels used to construct mild solutions via the Duhamel formula. 
The key properties of these kernels—such as self-similarity, integrability, and regularity—form the basis of our analysis of the existence and qualitative behavior of solutions. 
This section also establishes several technical lemmas that serve as crucial tools for the estimates developed later.

Section~\ref{sec:proofs} is devoted to the proofs of the main results. 
We first establish a local existence result, stated in Theorem~\ref{thm:Loc-Exis}. 
We then address the nonexistence of global solutions by proving Theorem~\ref{thm:blowup}. 
Next, we prove the global existence result in Theorem~\ref{thm:glob-exis-1}, which relies on the smallness of the initial data $u_0$ in the critical norm and the smallness of the forcing term $\bw$ in a suitable $L^{r_c}$ norm. 
Finally, we present Theorem~\ref{thm:glob-exis-2}, establishing global existence under significantly weaker local smallness conditions. 
Unlike Theorem~\ref{thm:glob-exis-1}, this result allows $u_0$ and $\bw$ to be large in a global norm, provided they satisfy appropriate local smallness and decay assumptions.

In Section~\ref{sec:num}, we complement the analytical results with numerical simulations of a one-dimensional fractional diffusion equation.
The simulations are designed to illustrate the critical threshold predicted by Theorem~\ref{thm:blowup} and to visualize the qualitative differences between subcritical and supercritical regimes. 
In particular, we compare the time evolution and spatial profiles of solutions for exponents below and above the critical value, highlighting the concentration mechanism leading to blow-up in the subcritical case and the bounded behavior observed in the supercritical regime.
These computations do not constitute a proof, but they provide a concrete and intuitive illustration of the theoretical phenomena established in the paper.

Throughout the remainder of the article, the constant $C$ may denote different values in different contexts.

\section{Useful tools \& Auxiliary results}
\label{sec:prelim}
In this section, we introduce the notation used throughout the paper and present several auxiliary results and estimates.

We begin with a brief introduction to the Caputo fractional derivative. For a comprehensive treatment of fractional calculus, we refer the reader to \cite{Diethelm, KST, Podlubny1999, SKM}.
The \emph{Riemann--Liouville fractional integral} of order $\alpha \in (0,1)$ is defined by
\begin{equation}  
\label{RL-Int}  
\mathcal{I}^{\alpha} u(t) = \frac{1}{\Gamma(\alpha)} \int\limits_0^t (t-s)^{\alpha-1} u(s) \, ds,  
\end{equation}  
where $\Gamma$ denotes the Gamma function, given by
\begin{equation*}
\Gamma(z) = \int\limits_0^\infty s^{z-1} e^{-s} \, ds, \quad \text{for } \Re(z) > 0.
\end{equation*}

The \emph{Riemann--Liouville fractional derivative} of order $\alpha \in (0,1)$ is then defined as
\begin{equation}  
\label{RL-FD}  
\cd^{\alpha} u(t) = \frac{d}{dt} \left( \mathcal{I}^{1-\alpha} u(t) \right) = \frac{1}{\Gamma(1-\alpha)} \frac{d}{dt} \left( \int\limits_0^t (t-s)^{-\alpha} u(s) \, ds \right),
\end{equation}  
provided that $t \mapsto \mathcal{I}^{1-\alpha} u(t)$ is absolutely continuous.

The \emph{Caputo fractional derivative} of order $\alpha \in (0,1)$ is defined by
\begin{equation}  
\label{Cap-FD1}  
\partial_t^{\alpha} u(t) = \cd^{\alpha} \left( u(t) - u(0) \right),
\end{equation}  
where $u(0)$ exists and $t \mapsto \mathcal{I}^{1-\alpha} u(t)$ is absolutely continuous.

Under the stronger assumption that $u$ is absolutely continuous on $[0,T]$, the Caputo derivative admits the following equivalent representation (see \cite[Theorem 2.1, p.~92]{KST}):
\begin{equation}  
\label{Cap-FD2}  
\partial_t^{\alpha} u(t) = \frac{1}{\Gamma(1-\alpha)} \int\limits_0^t (t-s)^{-\alpha} \frac{du}{ds}(s) \, ds.
\end{equation}
We now present the following lemma concerning the fractional integration-by-parts formula for the Caputo derivative.
\begin{lem}
\label{IPP}
Let $\alpha \in (0,1)$, $u \in AC([0,T])$, and $v \in C^1([0, T])$ with $v(T) = 0$. Then,
\begin{equation}
\int\limits_0^T v(t) \, \partial_t^\alpha u(t)\, dt = \int\limits_0^T \big( u(t) - u(0) \big) \, \cd_{t|T}^{\alpha} v(t)\, dt,
\end{equation}
where $\cd_{t|T}^{\alpha}$ is the right-sided Riemann-Liouville fractional derivative defined by
\begin{equation}
\label{RL-right}
\cd_{t|T}^{\alpha}v(t) = -\frac{1}{\Gamma(1-\alpha)} \frac{d}{dt} \int\limits_{t}^{T} (s-t)^{-\alpha} v(s)\, ds.
\end{equation}
\end{lem}
\begin{proof}
The result follows directly from the standard properties of the Caputo derivative and Riemann-Liouville fractional derivatives; see, for instance,  \cite[Section~2.4]{Podlubny1999}.

Recall that for $0<\alpha<1$, the Caputo derivative of order $\alpha$ is defined as
Thus,
$$
\int_0^T v(t) \, \partial_t^\alpha u(t)\, dt 
= \frac{1}{\Gamma(1-\alpha)} \int_0^T v(t) \left( \int_0^t (t-s)^{-\alpha} u'(s)\, ds \right) dt.
$$
By Fubini's theorem, we interchange the order of integration:
$$
= \frac{1}{\Gamma(1-\alpha)} \int\limits_0^T u'(s) \left( \int\limits_s^T (t-s)^{-\alpha} v(t)\, dt \right) ds 
= \int\limits_0^T u'(s) \, I^{1-\alpha}_{s|T} v(s)\, ds,
$$
where $ I^{1-\alpha}_{s|T} v(s) = \frac{1}{\Gamma(1-\alpha)} \int\limits_s^T (t-s)^{-\alpha} v(t)\, dt $.

We integrate by parts, we write
$$
\int_0^T u'(s) \, I_{s|T}^{1-\alpha} v(s)\, ds 
= \Big[ u(s) \, I_{s|T}^{1-\alpha} v(s) \Big]_{s=0}^{s=T} 
- \int_0^T u(s) \, \frac{d}{ds} I_{s|T}^{1-\alpha} v(s)\, ds.
$$
Now compute the derivative of the fractional integral. By the definition of the right-sided Riemann-Liouville derivative (\ref{RL-right}), we have
$$
\frac{d}{ds} I_{s|T}^{1-\alpha} v(s) 
= \frac{1}{\Gamma(1-\alpha)} \frac{d}{ds} \int_s^T (t-s)^{-\alpha} v(t)\, dt 
= - \cd_{s|T}^{\alpha} v(s).
$$
Evaluate the boundary term. Since $I_{T|T}^{1-\alpha} v(T) = 0$ (the integral is over an interval of zero length), we obtain
$$
\Big[ u(s) \, I_{s|T}^{1-\alpha} v(s) \Big]_{s=0}^{s=T} 
= u(T) \cdot 0 - u(0) \, I_{0|T}^{1-\alpha} v(0) 
= - u(0) \, I_{0|T}^{1-\alpha} v(0).
$$
Therefore,
$$
\int_0^T v(t) \, \partial_t^\alpha u(t)\, dt 
= - \int_0^T u(s) \, \frac{d}{ds} I_{s|T}^{1-\alpha} v(s)\, ds - u(0) I_{0|T}^{1-\alpha} v(0)
= \int_0^T u(s) \, \cd_{s|T}^{\alpha} v(s)\, ds - u(0) I_{0|T}^{1-\alpha} v(0).
$$
It remains to express $I_{0|T}^{1-\alpha} v(0)$ in terms of $\cd_{t|T}^{\alpha} v(t)$. Compute the integral of this derivative:
$$
\int_0^T \cd_{t|T}^{\alpha} v(t)\, dt 
= -\frac{1}{\Gamma(1-\alpha)} \int_0^T \frac{d}{dt} \left( \int_t^T (s-t)^{-\alpha} v(s)\, ds \right) dt.
$$

By the Fundamental Theorem of Calculus,
$$
= -\frac{1}{\Gamma(1-\alpha)} \left[ \int_t^T (s-t)^{-\alpha} v(s)\, ds \right]_{t=0}^{t=T}.
$$

For $t = T$, the integral $\int_T^T (s-T)^{-\alpha} v(s)\, ds = 0$. For $t = 0$, we obtain $\int_0^T s^{-\alpha} v(s)\, ds$. Thus,
$$
\int_0^T \cd_{t|T}^{\alpha} v(t)\, dt 
= \frac{1}{\Gamma(1-\alpha)} \int_0^T s^{-\alpha} v(s)\, ds 
= I_{0|T}^{1-\alpha} v(0).
$$
We deduce that
$$
- u(0) I_{0|T}^{1-\alpha} v(0) = - u(0) \int_0^T \cd_{t|T}^{\alpha} v(t)\, dt.
$$
Substituting back into the previous expression yields
$$
\int_0^T v(t) \, \partial_t^\alpha u(t)\, dt 
= \int_0^T u(t) \, \cd_{t|T}^{\alpha} v(t)\, dt - u(0) \int_0^T \cd_{t|T}^{\alpha} v(t)\, dt 
= \int_0^T \big( u(t) - u(0) \big) \, \cd_{t|T}^{\alpha} v(t)\, dt,
$$
which establishes the desired integration by parts formula.
\end{proof}
\subsection{Fundamental solution and Duhamel's formula}
A sufficiently smooth solution \( u \) to \eqref{main} can be represented via Duhamel's formula:
\begin{equation}
\begin{aligned}
\label{F-D}
u(t,x)
&= \underbrace{\int_{\mathbb{R}^N} u_0(y)\, Z(t, x-y)\, dy}_{\mathcal{L}_{u_0}(t,x)}
   + \underbrace{\int_{0}^{t} \int_{\mathbb{R}^N} |u(\tau,y)|^{p}\, Y(t-\tau, x-y)\, dy\, d\tau}_{\mathcal{N}_{u}(t,x)} \\
&\quad + \underbrace{\int_{0}^{t} \int_{\mathbb{R}^N} \tau^{\sigma}\, \bw(y)\, Y(t-\tau, x-y)\, dy\, d\tau}_{\mathcal{S}_{\bw}(t,x)}
   \; := \; \mathcal{G}u(t,x).
\end{aligned}
\end{equation}
Here \( Z \) denotes the fundamental solution of the homogeneous equation
$$
\partial_t^{\alpha} u + (-\Delta)^{\ts} u = 0.
$$
For $ \alpha \in (0,1)$, the associated kernel $ Y$ which corresponds to the time fractional Duhamel term satisfies $Y(t,\cdot) \sim \partial_t^\alpha Z(t,\cdot)$ in the sense of distributions \cite{CVPDE-2024, Kemp}.

More precisely, both \( Z \) and \( Y \) admit self-similar representations of the form
\[
Z(t,x) = t^{-\frac{\alpha N}{2\ts}}\, F\!\left(x\, t^{-\frac{\alpha}{2\ts}}\right),
\qquad
Y(t,x) = t^{-\frac{\alpha N}{2\ts}}\, G\!\left(x\, t^{-\frac{\alpha}{2\ts}}\right),
\]
where \( F \) and \( G \) are the associated spatial profiles.

We next recall several fundamental properties of the kernels \( Z \) and \( Y \).  
The following key lemma gathers the main estimates and structural features that will be crucial in our analysis.  
For proofs and additional details, we refer to \cite{Eid, Kemp, Kim, Koch, CVPDE-2024}.

\begin{lem}\label{lem:kernels}
Let $Z$ and $Y$ be the spatial kernels associated with the linear part of equation \eqref{main}, as defined in the integral formulation \eqref{F-D}. The following properties hold
~ \begin{enumerate}[label=(\roman*)]
    \item  Both kernels exhibit the scaling forms  
    \begin{equation*}
    Z(t,x) = t^{-\frac{N\alpha}{2\ts}} F\left(xt^{-\frac{\alpha}{2\ts}}\right), \quad Y(t,x) = t^{-\left(1-\alpha+\frac{N\alpha}{2\ts}\right)} G\left(xt^{-\frac{\alpha}{2\ts}}\right),
    \end{equation*}
    where $F, G$ are positive, radially decreasing functions, smooth away from the origin.    
  \item    $\|Z(t,\cdot)\|_{L^1(\mathbb{R}^N)} = 1$ for all $t > 0$.
        \item Let $q_c := \frac{N}{N-2\ts}$ if $N > 2\ts$, and $\infty$ if $N \leq 2\ts$. For any $q \in [1, q_c)$, $F \in L^q(\mathbb{R}^N)$ and 
        \begin{align}
        \label{Es-Z}
        \|Z(t,\cdot)\|_{L^q(\mathbb{R}^N)} = t^{-\frac{N\alpha}{2\ts}(1 - \frac{1}{q})}\|F\|_{L^q(\mathbb{R}^N)}, \quad t > 0.
        \end{align}
        \item For $t \geq t_0 > 0$, there exists $C = C(\alpha, \ts, N, q) > 0$ such that 
        \begin{equation*}
        \|Z(t,\cdot) - Z(t_0,\cdot)\|_{L^q(\mathbb{R}^N)} \leq C (t - t_0) t_0^{-1 - \frac{N\alpha}{2\ts}(1 - \frac{1}{q})},
        \end{equation*}
       \item  For any $\lambda > 0$ and $q \in [1, \infty]$,  
    \begin{equation*}
    \|Z(t,\cdot)\|_{L^q(|x| > \lambda t^{\frac{\alpha}{2\ts}})} = t^{-\frac{N\alpha}{2\ts}(1 - \frac{1}{q})}\|F\|_{L^q(|x| > \lambda)}.
    \end{equation*}
\item For every $q \in [1, q_c)$, the function $G$ belongs to $L^q(\mathbb{R}^N)$, and the following identity holds:
        \begin{align}
        \label{Es-Y}
        \|Y(t,\cdot)\|_{L^q(\mathbb{R}^N)} = t^{-( 1 - \alpha + \frac{N\alpha}{2\ts}\left(1 - \frac{1}{q}\right))}\|G\|_{L^q},  \quad t > 0.
        \end{align}
        \item There exists $C = C(\alpha, s, N, q) > 0$ such that for all $0 \leq \tau < t_0 \leq t$,   
        \begin{align}
        \label{Cont_Y}
        \|Y(t-\tau, \cdot) - Y(t_0-\tau, \cdot)\|_{L^q(\mathbb{R}^N)} \leq C (t - t_0)(t_0 - \tau)^{\alpha-\frac{N\alpha}{2\ts}(1-\frac{1}{q})-2}.
        \end{align}
    \end{enumerate}
 \end{lem}
 \begin{rem}
\rm A direct computation shows that
\[
1 - \alpha + \frac{N\alpha}{2\ts}\!\left(1 - \frac{1}{q}\right) < 1
\quad \Longleftrightarrow \quad q < q_c,
\]
which implies that
\[
Y(t,\cdot) \in L^{1}_{\mathrm{loc}}(0,\infty; L^{q})
\quad \text{if and only if} \quad 1 \le q < q_c.
\]
 \end{rem}
 A straightforward consequence of the kernel estimates \eqref{Es-Y} and \eqref{Cont_Y} is the following estimate.
\begin{lem}\label{ES_H} 
Let $\alpha \in (0,1)$ and $\gamma \in (0,1)$. There exists a constant $C > 0$ such that for all $0 < t_0 \leq t $, the following H\"older-type estimate holds:
\begin{align*}
\|Y(t,\cdot) - Y(t_0,\cdot)\|_{L^1(\mathbb{R}^N)} \leq C (t - t_0)^\gamma  t_0^{\alpha - 1 - \gamma}.
\end{align*}
\end{lem}
\begin{proof}
Let $0 < t_0 \leq t $ and $\gamma \in (0,1)$. From the time continuity estimate \eqref{Cont_Y} with $\tau = 0$ and $q = 1$, we have
\begin{equation*}
\|Y(t,\cdot) - Y(t_0, \cdot)\|_{L^1(\mathbb{R}^N)} \leq C (t - t_0) t_0^{\alpha - 2}.
\end{equation*}
Moreover, by the bound \eqref{Es-Y} with $q = 1$, we get
\begin{equation*}
\|Y(t,\cdot) - Y(t_0,\cdot)\|_{L^1(\mathbb{R}^N)} \leq \|Y(t,\cdot)\|_{L^1} + \|Y(t_0,\cdot)\|_{L^1} \leq C(t^{\alpha-1} + t_0^{\alpha-1}) \leq C t_0^{\alpha - 1},
\end{equation*}
where we used $t \geq t_0> 0$ and $\alpha - 1 < 0$. Interpolating between these two estimates yields
\begin{align*}
\|Y(t,\cdot) - Y(t_0,\cdot)\|_{L^1(\mathbb{R}^N)} \leq C\left[ (t - t_0) t_0^{\alpha - 2}\right]^\gamma  \left[t_0^{\alpha - 1}\right]^{1-\gamma} = C (t - t_0)^\gamma t_0^{\alpha - 1 - \gamma}.
\end{align*}   
\end{proof}

For a parameter \( \ts \in (0,1) \) and a function \( u : \mathbb{R}^N \to \mathbb{R} \), an equivalent definition of the fractional Laplacian of order \( \ts \) at a point \( x \in \mathbb{R}^N \) is given by
\begin{align*}
(-\Delta)^{\ts} u(x)
    = C_{N,\ts} \, \lim_{\varepsilon \to 0^{+}}
      \int_{\mathbb{R}^N \setminus B_{\varepsilon}(x)}
      \frac{u(x) - u(y)}{|x-y|^{N+2\ts}}\, dy,
\end{align*}
whenever the limit exists. Here \( B_{\varepsilon}(x) \) denotes the ball of radius \( \varepsilon \) centered at \( x \), and
\[
C_{N,\ts}
    = \frac{2^{\,2\ts - 1}\, \Gamma\!\left( \frac{N}{2} + \ts \right)}
           {\pi^{N/2}\, \Gamma(1 - \ts)}
\]
is the normalization constant ensuring that \( (-\Delta)^{\ts} \) converges to the classical Laplacian as \( \ts \to 1^{-} \). Further equivalent formulations and additional properties of the fractional Laplacian can be found in \cite{Book-FC, M, Guide}.

We now recall a fundamental symmetry property that plays a central role in fractional integration by parts.

\begin{lem} \label{I-F}
Let \( \ts \in (0,1) \). For any functions \( u, v \in C^{2}(\mathbb{R}^{N}) \cap L^{\infty}(\mathbb{R}^{N}) \), the fractional Laplacian satisfies the identity
\begin{align*}
\int_{\mathbb{R}^N} u(x)\, (-\Delta)^{\ts} v(x)\, dx
    = \int_{\mathbb{R}^N} v(x)\, (-\Delta)^{\ts} u(x)\, dx.
\end{align*}
\end{lem}
%%%%%%%%%%%%%%%%%%%%%%%%%%%%%%%%%%%%%%%%%%
We also recall the following form of the singular Gronwall inequality.

\begin{lem}\label{lem:singular-Gronwall}
Let \( T > 0 \), \( \gamma > -1 \), and \( K > 0 \).  
Assume that \( \varphi \in C([0,T]; \mathbb{R}_+) \) satisfies
\begin{align*}
\varphi(t)
    \le K \int_{0}^{t} (t-\tau)^{\gamma} \varphi(\tau)\, d\tau,
    \qquad \text{for all } t \in [0,T].
\end{align*}
Then \( \varphi(t) = 0 \) for all \( t \in [0,T] \).
\end{lem}

We now establish global existence in Lebesgue spaces \( L^{q} \) under suitable assumptions on the parameters \( p, q \), together with a smallness condition.

\begin{lem}\label{lem:W-Condition}
Assume \( \alpha, \ts \in (0,1) \), \( -\alpha < \sigma < 0 \), and \( 2\ts < N \).  
Suppose
\begin{align}
\label{C-p}
\frac{N\alpha - 2\ts\sigma}{N\alpha - 2\ts(\alpha + \sigma)}
    \leq p < \frac{N\alpha - 2\ts q\, \sigma}{N\alpha - 2\ts q\, (\alpha + \sigma)},
\end{align}
and let \( q \) belong to the admissible range
\[
p_c < q \leq p\, p_c, 
\qquad q \geq p,
\]
where \( p_c \) is given in \eqref{p-c}.  
Let \( u \) be a maximal \( L^{q} \)-solution of \eqref{main} corresponding to initial data  
\( u_0 \in L^{q}(\mathbb{R}^N) \), and assume \( \bw \in L^{q}(\mathbb{R}^N) \). Then there exists a constant \( \mathscr{M} > 0 \), depending only on  
\( p, q, \alpha, \ts, N \), and \( \sigma \), such that if
\begin{equation}\label{cond:smallness}
t^{\,\frac{\alpha}{p-1} - \frac{N\alpha}{2\ts q}}
\Bigl(
    \|\mathcal{L}_{u_0}(t,\cdot)\|_{L^{q}}
    + \|\mathcal{S}_{\bw}\|_{L^{q}}
\Bigr)
    \leq \mathscr{M},
    \qquad \forall\, t > 0,
\end{equation}
then the following an a priori bound:
\begin{equation}\label{est:apriori}
\sup_{\tau \in (0,t)}
\tau^{\,\frac{\alpha}{p-1} - \frac{N\alpha}{2\ts q}}
\|u(\tau,\cdot)\|_{L^{q}}
    \leq 2\mathscr{M},
\qquad \forall\, t \in (0, T_{\max}),
\end{equation}
where \( T_{\max} \in (0,\infty] \) denotes the maximal existence time.
In particular, \( u \) is a global \( L^{q} \)-solution of \eqref{main}.
\end{lem}
\begin{proof}[Proof of Lemma~\ref{lem:W-Condition}]  First, we establish the existence of a maximum time $T_{\text{max}} > 0$ and a maximal solution
$u \in C\big([0, T_{\text{max}}); L^q(\mathbb{R}^N)\big)$
satisfying the integral equation \eqref{F-D}. From the results of \cite{CVPDE-2024}, we already know that
$\mathcal{L}_{u_0},\mathcal{N}_u \in C([0,T]; L^q)$. To conclude that the full solution $u$ belongs to $C([0,T]; L^q)$, it remains to prove the continuity of the term $S_{\bw}(t)$.
We analyze the continuity at $t=0$ and for $t>0$ separately.

Let $\alpha \in (0,1)$, $\sigma + \alpha > 0$, and $\bw \in L^q(\mathbb{R}^N)$. Using \eqref{Es-Y} and Young's inequality, we get
\begin{align}
\label{Continuity $t = 0$}
\|\mathcal{S}_\bw(t)\|_{L^q} \leq \|\bw\|_{L^q} \int\limits_0^t \tau^\sigma \|Y(t-\tau)\|_{L^1}  ds \leq C \|\bw\|_{L^q} \int\limits_0^t \tau^\sigma (t-\tau)^{\alpha-1}  d\tau.
\end{align}
The integral is equal to $t^{\sigma+\alpha} \mathscr{B}(\sigma+1, \alpha)$. Since $\sigma + \alpha > 0$, we have
$$
\|\mathcal{S}_\bw(t)\|_{L^q} \leq C \|\bw\|_{L^q} t^{\sigma+\alpha} \to 0 \quad \text{as } t \to 0^+.$$
On the other hand, let $0 \leq t_1 < t_2 \leq T$. We decompose
$$\mathcal{S}_\bw(t_2) - \mathcal{S}_\bw(t_1) = \int\limits_{t_1}^{t_2} \tau^\sigma (Y(t_2-\tau) * \bw)  d\tau + \int\limits_0^{t_1} \tau^\sigma \left(Y(t_2-\tau) - Y(t_1-\tau)\right) * \bw  d\tau.$$
For the first term, an application of Minkowski's inequality together with the kernel estimate from Lemma~\ref{lem:kernels}, while taking into account that $\tau \in [t_1, t_2] \subset [0, T]$, yields
\begin{align}
\label{Es-1}
\left\| \int\limits_{t_1}^{t_2} \tau^\sigma Y(t_2 - \tau) * \bw  d\tau \right\|_{L^q} &\leq \| \bw \|_{L^q} \int\limits_{t_1}^{t_2} \tau^\sigma \| Y(t_2 - \tau) \|_{L^1}  d\tau \\ 
&\leq C \| \bw \|_{L^q} \int\limits_{t_1}^{t_2} \tau^\sigma (t_2 - \tau)^{\alpha - 1} d\tau &\leq \frac{C t_1^{\sigma}}{\alpha} \| \bw \|_{L^q} (t_2 - t_1)^\alpha.
\end{align}
We now turn to the second term. By applying Lemma \ref{ES_H} with the choices $t_0 = t_1 - \tau$, $t = t_2 - \tau$, and with $\gamma$ restricted to $(0,\alpha+\sigma)$, we obtain 
\begin{align}
\label{Es-2}
\left\| \int\limits_0^{t_1} \tau^\sigma \left[ Y(t_2 - \tau,.) - Y(t_1 - \tau,.) \right] * \bw  d\tau \right\|_{L^q} &\leq C_\gamma \| \bw \|_{L^q} (t_2 - t_1)^\gamma \int\limits_0^{t_1} \tau^\sigma (t_1 - \tau)^{\alpha - 1 - \gamma}  d\tau\\
&\nonumber \leq C_\gamma T^{\alpha+\sigma-\gamma} \mathscr{B}(\sigma+1,\alpha-\gamma)\|\bw\|_{L^q} (t_2-t_1)^{\gamma}
\end{align}
Combining \eqref{Es-1} and \eqref{Es-2}, we conclude that
$$
\|\mathcal{S}_\bw(t_2) - \mathcal{S}_\bw(t_1) \|_{L^q} \leq C \| \bw \|_{L^q} \left( t_1^\sigma|t_2 - t_1|^\alpha + |t_2 - t_1|^\gamma \right).
$$
where C depends on $\gamma$, $\alpha$, and $T$.
This bound tends to zero as $|t_2 - t_1| \to 0$, which establishes that $\mathcal{S}_\bw \in C([0,T]; L^q)$.
Since we have shown that $u \in C([0, T_{\max}); L^q)$, it follows that the limit $u(\cdot, T_{\max}) := \lim_{t \nearrow T_{\max}} u(\cdot, t)$ exists in the $L^q(\mathbb{R}^N)$ .

Subsequently, we prove that if $T_{\text{max}} < \infty$, then
$\limsup_{t \nearrow T_{\text{max}}} \|u(\cdot, t)\|_{L^q} = +\infty.$
We proceed by contradiction, assuming that $T := T_{\text{max}} < \infty$ but that
$\limsup_{t \nearrow T} \|u(\cdot, t)\|_{L^q} < +\infty,$ then we need to prove first an extension result that will characterize the behavior of the maximal solution at the boundary of it's existence interval.

Let $h$ be such that $0 < h < T_{\text{max}}/2$. For any $t \in (T_{max} - h, T_{max} + h)$, we define
\[
\begin{aligned}
\label{S-M}
\mathcal{H}v(t,x) = &\underbrace{\int\limits_{\mathbb{R}^N} u_0(y) Z(t,x-y)  dy}_{\mathcal{L}_{u_0}(t,x)} 
+ \underbrace{\int\limits_0^{T_{max}-h} \int\limits_{\mathbb{R}^N} |u(\tau,y)|^{p} Y(t-\tau,x-y)  dy  d\tau}_{\mathcal{N}_u(t,x)} \\
&+ \underbrace{\int\limits_{T_{max}-h}^t \int\limits_{\mathbb{R}^N} |v(\tau,y)|^{p} Y(t-\tau,x-y)  dy  d\tau}_{\mathcal{F}_v(t,x)} 
+ \underbrace{\int\limits_0^t \int\limits_{\mathbb{R}^N} \tau^{\sigma} \bw(y) Y(t-\tau,x-y)  dy  d\tau}_{\mathcal{S}_{\bw}(t,x)}.
\end{aligned}
\]
We establish a refined decomposition of $\mathcal{S}_\bw(t)$ that depends on the position of the parameter $t$
\begin{align}
\label{Decomp-S}
\mathcal{S}_{\bw}(t) = \mathcal{S}_{1,\bw} + \mathcal{S}_{2,\bw}(t) + \mathcal{S}_{3,\bw}(t) + \mathcal{S}_{4,\bw}(t),
\end{align} \text{where}
\[
\begin{aligned}
\mathcal{S}_{1,\mathbf{\bw}} &= \int\limits_0^{T_{max} -h} \tau^\sigma \left( Y(T_{max} - \tau,.) * \bw \right)  d\tau, \\
\mathcal{S}_{2,\mathbf{\bw}}(t) &= \int\limits_0^{T_{max} - h} \tau^\sigma \left( Y(t - \tau,.) - Y(T_{max} - \tau,.) \right) * \bw  d\tau, \\
\mathcal{S}_{3,\mathbf{\bw}}(t) &= \int\limits_{T_{max} - h}^{\min(t,T_{max})} \tau^\sigma \left( Y(t - \tau,.) * \bw \right)  d\tau, \\
\mathcal{S}_{4,\mathbf{\bw}}(t) &= \int\limits_{T_{max}}^t \tau^\sigma \left( Y(t - \tau,.) * \bw \right)  d\tau, \text{for} \quad t > T_{max}.
\end{aligned}
\]
We construct a solution $v$ that coincides with \eqref{S-M} on the interval $[0, T_{\max} - h)$, where $\tau \in (0, T_{\max}/2)$ is sufficiently small. For positive parameters $h>0$ and $\delta>0$, we work in the space
$$\mathcal{X}_{h,\delta} = \left\{ v \in L^\infty\big( (T_{max} - h, T_{max} + h); L^q(\mathbb{R}^N) \big) :\sup_{t \in [T_{max}-h,T_{max}+h]} \|v(t,\cdot)\|_{L^q} \leq \delta \right\}.$$
Under the assumptions on $\sigma$ and $\alpha$, applying the kernel estimates \eqref{Es-Y}, \eqref{Cont_Y}, for $Y$ and $Z$ together with Young's inequality and the $\gamma$- H\"older regularity estimate \eqref{ES_H} yields
\begin{align}
\label{Es-S_1}    
\| \mathcal{S}_{1,\bw} \|_{L^q} &\leq \| \bw \|_{L^q} \int\limits_0^{T_{max} - h} \tau^\sigma \| Y(T_{max} - \tau,.) \|_{L^1}  d\tau \\
&\nonumber \leq C \| \bw \|_{L^q} \int\limits_0^{T_{max} - h} \tau^\sigma (T_{max} - \tau)^{\alpha - 1}  d\tau \leq C \| \bw \|_{L^q} T_{max}^{\sigma + \alpha} \mathscr{B}(\sigma + 1, \alpha) =: \mathscr{K}_1.
\end{align}
The estimate for $\mathcal{S}_{2,\bw}$ follows using analogous techniques. Under the hypothesis $0 < \gamma < \alpha$, Lemma \ref{ES_H} gives
\begin{align}
\label{Es-S_2}    
\|\mathcal{S}_{2,\bw}\|_{L^q} &\leq \|\bw\|_{L^q} \int\limits_0^{T_{max}-h} \tau^\sigma \|Y(t-\tau,.) - Y(.,T_{max}-\tau)\|_{L^1}  d\tau\\&  \leq C_\gamma \|\bw\|_{L^q}  h^\gamma \int\limits_0^{T_{max}-h} \tau^\sigma (T_{max}-\tau)^{\alpha-1-\gamma}  d\tau \\
&\leq C_\gamma \|\bw\|_{L^q}  h^\gamma  T_{max}^{\sigma+\alpha-\gamma}  \mathscr{B}(\sigma+1, \alpha-\gamma) =: \mathscr{K}_2  h^\gamma.
\end{align}
For $\tau \in [T_{max} -h, \min(t,T_{max})]$, we have $\tau \geq T_{max} - h \geq T_{max}/2$, so $\tau^\sigma \leq (T_{max}/2)^\sigma$
\begin{align}
\|\mathcal{S}_{3,\bw}\|_{L^q} 
&\leq \|\bw\|_{L^q} \int\limits_{T_{max}-h}^{\min(t,T_{max})} \tau^\sigma \|Y(t-\tau,.)\|_{L^1}  d\tau \\ \label{Es-S_3}
&\leq C\|\bw\|_{L^q} \int\limits_{T_{max}-h}^{\min(t,T_{max})} \tau^\sigma (t-\tau)^{\alpha-1} d\tau \\& C \|\bw\|_{L^q} T_{max}^\sigma h^\alpha =: \mathscr{K}_3 h^\alpha.
\end{align}
Similarly, for $t > T_{max}$ and  $\tau \in [T_{max}, t]$, we have $\tau \geq T_{max} \geq T_{max}/2 \Rightarrow \tau^\sigma \leq (T_{max}/2)^\sigma$, and 
\begin{align}
\label{Es-S_4}
\|\mathcal{S}_{4,\bw}\|_{L^q} &\leq \|\bw\|_{L^q} \int\limits_{T_{max}}^t \tau^\sigma \|Y(t - \tau,.)\|_{L^1}  d\tau \\
&\nonumber \leq C \|\bw\|_{L^q} \int\limits_{T_{max}}^t \tau^\sigma (t - \tau)^{\alpha - 1}  d\tau \leq C T_{max}^{\sigma} h^{\alpha} \|\bw\|_{L^q} =: \mathscr{K}_4 h^\alpha.
\end{align}
We combine the estimates  \eqref{Es-S_1}, \eqref{Es-S_2}, \eqref{Es-S_3}, and \eqref{Es-S_4}. By choosing $h$ small enough so that
$$\mathscr{K}_2 h^\gamma + (\mathscr{K}_3 + \mathscr{K}_4)h^\alpha \leq \frac{\mathscr{K}_1}{2}.$$
Based on the estimates for $\mathcal{L}_{u_0}$, $\mathcal{N}_u$, and $\mathcal{F}_v$ obtained in \cite{CVPDE-2024}, there exists, for all sufficiently small $\mathscr{K}$, a constant such that
$$\|\mathcal{H}v\|_{L^q} \leq \|\mathcal{L}_{u_0}(t) + \mathcal{N}_u(t) + \mathcal{F}_v(t)\|_{L^q}+\|\mathcal{S}_\bw\|_{L^q} \leq \mathscr{K}.$$
Next, we establish the contractivity of $\mathcal{G}$ on $X_{h,\delta}$ for sufficiently small $h$. Specifically, for any $u, v \in X_{h,\delta}$ and $t \in [T_{\max}-h, T_{\max} + h]$,
\begin{align*}
\|\mathcal{G}(u)(t) - \mathcal{G}(v)(t)\|_{L^q} 
&\leq \int\limits_{T_{\max}-h}^t \left\| |u(\tau)|^{p-1}u(\tau) - |v(\tau)|^{p-1}v(\tau) \right\|_{L^{q/p}} \|Y(t-\tau,\cdot)\|_{L^r} \, d\tau \\
&\leq C \int\limits_{T_{\max}-h}^t \|u(\tau) - v(\tau)\|_{L^q} \left( \|u(\tau)\|_{L^q}^{p-1} + \|v(\tau)\|_{L^q}^{p-1} \right) (t-\tau)^{\alpha-1-\frac{\alpha N}{2\ts}\left(1-\frac{1}{r}\right)} \, d\tau\\
&\leq 2C\delta^{p-1} \|u - v\|_{L^\infty(T_{\max}-h,T_{\max}+h;L^q)} \int\limits_{T_{\max}-h}^t (t-\tau)^{\alpha-1-\frac{\alpha N}{2\ts}\left(1-\frac{1}{r}\right)} \, d\tau\\
&\leq C\delta^{p-1} h^{\alpha-\frac{\alpha N}{2\ts}\left(\frac{p-1}{q}\right)} \|u - v\|_{L^\infty(T_{\max}-h,T_{\max}+h;L^q)},
\end{align*}
Choosing $h$ small enough such that
$$ C\delta^{p-1} h^{\alpha-\frac{\alpha N}{2\ts}\left(\frac{p-1}{q}\right)}\leq\frac{1}{2}.$$
Hence, $\mathcal{G}$ is a contraction on $X_{h,\delta}$, and therefore admits a unique fixed point $v \in X_{h,\delta}$. Moreover, since $
v(T_{\max}) = \mathcal{G}(v)(T_{\max}) = u(T_{\max}),$
the function $v$ defines the desired extension:
$$
v(t,x) = \begin{cases} 
u(t,x), & t \in [0, T_{\max} - h), \\
\tilde{u}(t,x), & t \in [T_{\max} - h, T_{\max} + h).
\end{cases}
$$
Since at this stage $\|u\|_{L^q}$ remains bounded, we can extend the solution to the interval $[0, T_{\max} + h)$. 
Assume, by contradiction, that $T_{\max} < \infty$ and
$$
\sup_{t \in [0, T_{\max})} \|u(t)\|_{L^q} \leq M < \infty.
$$
From the continuity result established earlier, the mapping $t \mapsto u(\cdot, t)$ is continuous in $L^q(\mathbb{R}^N)$ on $[0, T_{\max})$. Therefore, the limit
$$
u(T_{\max}, \cdot) := \lim_{t \to T_{\max}^{-}} u(t,\cdot)
$$
exists in $L^q(\mathbb{R}^N)$, and in particular, $u(\cdot, T_{\max}) \in L^q(\mathbb{R}^N)$.

By the local existence and extension result for mild solutions, if the solution $u$ belongs to $L^q(\mathbb{R}^N)$ at time $T_{\max}$, then it can be extended to a larger time interval $[0, T_{\max} + h)$, for some $h > 0$, while remaining a solution in the same class.
This contradicts the maximality of $T_{\max}$.
Hence, the initial assumption is false, and we must have:
\begin{equation}
\text{If } T_{\max} < \infty, \text{ then } \limsup_{t \to T_{\max}^{-}} \|u(t)\|_{L^q} = +\infty.
\label{eq:blowup}
\end{equation}
On the other hand, we have
\begin{align*}
 \sup_{\tau \in(0,t)} \tau^\beta \|u(\tau,\cdot)\|_{L^q} &\leq \sup_{\tau \in(0,t)} \tau^\beta \|\mathcal{L}_{u_0}\|_{L^q} + \sup_{\tau \in(0,t)} \tau^\beta \|\mathcal{N}_u\|_{L^q} + \mathscr{M}_2  \sup_{\tau \in(0,t)}\tau^{\beta + \sigma + \alpha} \\
&\leq \mathscr{M}_0 + \mathscr{M}_1 \left(\sup_{\tau \in (0,t)} \|\tau^\beta u(\tau, \cdot)\|_{L^q}\right)^p + \mathscr{M}_2 t^{\beta + \sigma + \alpha}   
\end{align*}
where $\mathscr{M}_1 = \|G\|_{L^r} \int\limits_0^1 \tau^{-\beta p}(1-\tau)^{-(1 - \alpha + \frac{N\alpha}{2\ts q}(p-1))} d\tau$, and $\mathscr{M}_2  = \|\bw\|_{L^q} \|G\|_{L^1} \mathscr{B}(\sigma+1, \alpha).$ 

Under the conditions \eqref{C-p} and $0 < \beta < \frac{1}{p}$, which ensure the convergence of the integral $\mathscr{M}_1$, and under the hypothesis $\sigma > -\alpha$, we obtain a result that, through a bootstrap argument, yields
$$\sup_{\tau \in (0,t)} \|\tau^\beta u(\tau, \cdot)\|_{L^q} \leq 2\mathscr{M} \quad \text{for all} \quad t \in (0, T).$$
Therefore, we deduce that
$$\|u(t, \cdot)\|_{L^q(\mathbb{R}^N)} \leq 2\mathscr{M} t^{-\beta} \quad \text{for all } t \in (0, T).$$
In particular, for any $\tau \in (0, T)$
$$\sup_{t \in [0, \tau]} \|u(t, \cdot)\|_{L^q(\mathbb{R}^N)} \leq \max\left\{ \|u_0\|_{L^q(\mathbb{R}^N)},\ \sup_{\tau \in(0,T)} 2\mathscr{M} \tau^{-\beta} \right\} < \infty,$$
which contradicts \eqref{eq:blowup}. Therefore, $T_{\max} = \infty$, completing the proof.
\end{proof}
 \section{Proofs of the main results}
 \label{sec:proofs}
 \subsection{Proof of Theorem~\ref{thm:Loc-Exis}}
\label{subsec:Loc-exis}
For given $T>0$, and $u_0,\bw \in L^q(\mathbb{R}^N)$, and let $\delta$ be a sufficiently small positive constant. We define the function space  
$$Y = \left\{ u \in C\big((0,T), L^q(\mathbb{R}^N)\big), \left\|u\right\|_{L^\infty((0,T), L^q({\mathbb{R}^N})} \leq \delta \right\},$$ 
equipped with the metric $d(u,v) = \sup_{t\in[0,T]} \|u - v\|_{L^q(\mathbb{R}^N)}$, for $u, v \in Y$. Since $(Y,d)$ is a complete metric space.  

Let $ \alpha, \ts \in (0,1),$  $p > 1$, $\sigma > -1$, suppose $u_0,\bw \in L^{q}(\mathbb{R}^N)$ satisfy $
\|u_0\|_{L^{q}} + \|\bw\|_{L^{q}} < \varepsilon$
for some $\varepsilon > 0$. We define the operator $\mathcal{G}$ on $Y$ by
\begin{equation}
\label{G-Y}
\begin{split}
\mathcal{G}u(t,x) &=\int\limits_{\mathbb{R}^N} u_0(y) Z(t,x-y) dy+ \int\limits_0^t \int\limits_{\mathbb{R}^N} |u(\tau,y)|^{p} Y(t-\tau,x-y) dy d\tau\\&+\int\limits_0^t \int\limits_{\mathbb{R}^N} \tau^{\sigma} \bw(y) Y(t-\tau,x-y) dy d\tau.
\end{split}
\end{equation}
Suppose that $u \in C([0,T], L^q(\mathbb{R}^N))$ with $q > p_c$. It is then straightforward to verify that \eqref{G-Y} also belongs to $C([0,T], L^q(\mathbb{R}^N))$.

Let $u_0, \bw \in L^{q}(\mathbb{R}^N)$. Using the estimates for the kernels $Z$ and $Y$ from Lemma~\ref{lem:kernels}, for $t \in [0, T ]$, we obtain
\begin{align*}
\|\mathcal{G}u(t)\|_{L^q} &\leq \|u_0\|_{L^q(\mathbb{R}^N)}+ \left\| \int\limits_0^t \int\limits_{\mathbb{R}^N} Y(t-\tau,x-y) |u(\tau,y)|^p dy d\tau \right\|_{L^q(\mathbb{R}^N)}\\&+\left\|\int\limits_0^t \int\limits_{\mathbb{R}^N} \tau^\sigma Y(t-\tau,x-y) \bw(y) dy d\tau\right\|_{L^q(\mathbb{R}^N)}\\
&\leq \|u_0\|_{L^q}+ C \delta^p \int\limits_0^t (t-\tau)^{-\frac{\alpha N}{2\ts}\left( \frac{p-1}{q} \right) + \alpha - 1} d\tau+C \| \bw \|_{L^q(\mathbb{R}^N)} \int\limits_0^t \tau^\sigma (t-\tau)^{\alpha-1} d\tau.
\end{align*}
When $N \geq 2\ts$, the condition $q > p_c$ ensures that the exponent satisfies $\alpha-\frac{\alpha N}{2\ts}\left( \frac{p-1}{q} \right) - 1 > -1$, which guarantees the convergence of the time integral. For $N < 2\ts$, this condition is automatically satisfied for $q \geq p$. For $t \in (0,T)$, the change of variable $\tau = t\theta$ gives 
\begin{align*}
\|\mathcal{G}u(t)\|_{L^q(\mathbb{R}^N)} &\leq \|u_0\|_{L^q(\mathbb{R}^N)}+ C \delta^p t^{\alpha - \frac{\alpha N}{2\ts}\left( \frac{p-1}{q} \right)} \int\limits_0^1 (1-\theta)^{\alpha - \frac{\alpha N}{2\ts}\left( \frac{p-1}{q} \right) - 1} d\theta\\&+ C t^{\alpha+\sigma} \| \bw \|_{L^q(\mathbb{R}^N)} \int\limits_0^1 \theta^{\sigma}(1-\theta)^{\alpha-1} d\theta \\
&\leq \|u_0\|_{L^q(\mathbb{R}^N)}+ C \delta^p T^{\alpha - \frac{\alpha N}{2\ts}\left( \frac{p-1}{q} \right)}+C T^{\alpha+\sigma} \mathscr{B}\left(\sigma+1,\alpha\right) \| \bw \|_{L^q(\mathbb{R}^N)}.
\end{align*}
Hence, we can choose $T$ small enough such that
$$C \delta^p T^{\alpha - \frac{\alpha N}{2\ts}\left( \frac{p-1}{q} \right)}+C T^{\alpha+\sigma} \mathscr{B}(\sigma+1,\alpha) \| \bw \|_{L^q(\mathbb{R}^N)}\leq \delta - \|u_0\|_{L^q}$$
This shows that $\mathcal{G}$ is a mapping from $Y$ into itself. It remains to demonstrate that the self-mapping $\mathcal{G}: Y \rightarrow Y$ is a contraction to complete the proof.
Applying similar arguments as mentioned above, we can establish that for functions $u$ and $v$ in $Y$
\begin{align*}
\|\mathcal{G}(u) - \mathcal{G}(v)\|_{L^q(\mathbb{R}^N)} &\leq C \int\limits_0^t (t - \tau)^{\alpha-1-\frac{\alpha N}{2\ts}\left(\frac{p-1}{q}\right)} \||u|^p - |v|^p\|_{L^{q/p}(\mathbb{R}^N)} \, ds \\
&\leq C\int\limits_0^t (t - \tau)^{\alpha-1-\frac{\alpha N}{2\ts}\left(\frac{p-1}{q}\right)} \left(\|u\|^{p-1}_{L^q}+\|v\|^{p-1}_{L^q}\right) \|u - v\|_{L^q(\mathbb{R}^N)} \, d\tau.
\end{align*}
where, for $p > 1$, we have used the inequality  
$$\left| |u|^p - |v|^p \right| \le C |u - v| \left( |u|^{p-1} + |v|^{p-1} \right).
$$
It follows that,
\begin{align*}
\|\mathcal{G}(u) - \mathcal{G}(v)\|_{L^q(\mathbb{R}^N)}&\leq C \delta^{p-1} \int\limits_0^t (t - \tau)^{\alpha-1-\frac{\alpha N}{2\ts}\left(\frac{p-1}{q}\right)} \, d\tau \|u - v\|_Y \\
&= C\delta^{p-1} T^{\alpha-\frac{\alpha N}{2\ts}\left(\frac{p-1}{q}\right)} \int\limits_0^1 (1 - \tau)^{\alpha-1-\frac{\alpha N}{2\ts}\left(\frac{p-1}{r}\right)} \, d\tau  \|u - v\|_Y.    
\end{align*}
By the Banach fixed point theorem, there exists a unique $u \in Y$ such that $u = \mathcal{G}u$, which is the desired solution.
The uniqueness is proven using a standard Gronwall argument. Consider two mild solutions 
$u, v \in C([0, T], L^q(\mathbb{R}^N))$ of \eqref{main}. Under the parameter constraints 
$q > p_c$ and $q \geq p$, the following key estimate holds:
\begin{align*}
\|\mathcal{G}u(t)-\mathcal{G}v(t)\|_{L^q(\mathbb{R}^N)}&=\| u(t) - v(t) \|_{L^q(\mathbb{R}^N)}\\& \leq C M \int\limits_0^t (t - \tau)^{\alpha - 1 - \frac{\alpha N(p-1)}{2sq}} \| u(\tau) - v(\tau) \|_{L^q(\mathbb{R}^N)}  d\tau,
\end{align*}
where $M = \sup_{t \in [0, T]} \left( \| u(t) \|_{L^q}^{p-1} + \| v(t) \|_{L^q}^{p-1} \right)$ is a finite constant. Since the exponent $$\alpha - 1 - \frac{\alpha N(p-1)}{2\ts q} > -1,$$ the conditions for the fractional Gronwall inequality \eqref{lem:singular-Gronwall} are satisfied , from which we conclude $u \equiv v$ on $[0, T]$.

\subsection{Proof of Theorem~\ref{thm:blowup}} Let $u$ be a global weak solution to problem \eqref{main}, so that $u \in (0, \infty)$. For given positive numbers $T$ and $R$ satisfying $T R^{2\ts/\alpha} < \infty$, we can choose $T^{*} > T R^{2\ts/\alpha}$.

To establish a contradiction, we proceed with a carefully chosen test function argument. We begin by introducing a smooth cutoff function $\psi \in \mathbf{C}_{0}^{\infty}([0,\infty))$ with the following properties
\begin{equation*}
0 \leqslant \psi \leqslant 1 \quad \text{and} \quad \psi(\tau) = 
\begin{cases} 
1 & \text{if } 0\leq \tau\leq 1, \\
0 & \text{if } \tau \geq 2.
\end{cases}
\end{equation*}
The core of our argument relies on the rescaled test function
$$
\psi_R(t,x) = \psi\left(\frac{t}{R^{2 \ts/\alpha}}\right)\psi\left(\frac{|x|}{R}\right), \quad R > 0.
$$
By multiplying problem \eqref{main} by $\psi_R$, applying the integration by parts formula from Lemmas \eqref{IPP} and \eqref{I-F},we exploit the properties of $\psi$, we derive the desired contradiction.
\begin{align}
\label{EQ}
&\int\limits_0^T \int\limits_{\mathbb{R}^N} u_0(x) \, \cd_{t|T}^{\alpha} \psi_R(t,x) \,dx\,dt + \int\limits_0^T \int\limits_{\mathbb{R}^N} |u|^p \psi_R(t,x) \,dx\,dt + \int\limits_0^T \int\limits_{\mathbb{R}^N} t^{\sigma} \bw \, \psi_R(t,x) \,dx\,dt \\
&\nonumber= \int\limits_0^T \int\limits_{\mathbb{R}^N} u (-\Delta)^{s} \psi_R(t,x) \,dx\,dt + \int\limits_0^T \int\limits_{\mathbb{R}^N} u \, \cd_{t|T}^{\alpha} \psi_R(t,x) \,dx\,dt.
\end{align}
For our analysis, we define the following key integral quantities $I=\int\limits_0^T \int\limits_{\mathbb{R}^N} |u| |(-\Delta)^{\ts} \psi_R(t,x)| \,dx\,dt,$ and
$J=\int\limits_0^T \int\limits_{\mathbb{R}^N}  |u| |\cd_{t|T}^{\alpha} \psi_R(t,x)| \,dx\,dt$. To derive suitable bounds for both $I$ and $J$, we apply the $\varepsilon$-Young inequality and obtain the estimates
\begin{align*}
I &\leq \frac{1}{2} \int\limits_0^T \int\limits_{\mathbb{R}^N}  |u|^p \psi_R(t,x) \,dx\,dt+C\int\limits_0^{TR^{\frac{2s}{\alpha}}} \psi\left(\frac{t}{R^{2 \ts/\alpha}}\right) \,dt \int\limits_{\mathbb{R}^N} \left|(-\Delta)^{\ts} \psi\left(\frac{|x|}{R}\right)\right|^{\frac{p}{p-1}} \left|\psi\left(\frac{|x|}{R}\right)\right|^{-\frac{1}{p-1}}\,dx \\
&\leq \frac{1}{2} \int\limits_0^T \int\limits_{\mathbb{R}^N}  |u|^p \psi_R(t,x) \,dx\,dt+C R^{N+\frac{2 \ts}{\alpha}-\frac{2\ts p}{p-1}}.
\end{align*}
Parallel reasoning applied to the second integral, using the Riemann-Liouville fractional derivative definition \eqref{RL-FD}, produces
\begin{align*}
J &\leq \frac{1}{2} \int\limits_0^{T} \int\limits_{\mathbb{R}^N}  |u|^p \psi_R(t,x) \,dx\,dt+ \int\limits_0^{TR^{\frac{2\ts}{\alpha}}} \left|\cd_{t|TR^{\frac{\ts}{\alpha}}}^{\alpha}\psi\left(\frac{t}{R^{\frac{2\ts}{\alpha}}}\right)\right|^{\frac{p}{p-1}}\psi\left(\frac{t}{R^{2\ts/\alpha}}\right)^{-\frac{1}{p-1}} \,dt \int\limits_{\mathbb{R}^N} \psi\left(\frac{|x|}{R}\right) \,dx \\
&\leq \frac{1}{2} \int\limits_0^T \int\limits_{\mathbb{R}^N}  |u|^p \psi_R(t,x) \,dx\,dt+C R^{N+\frac{2 \ts}{\alpha}-\frac{2\ts p}{p-1}}.
\end{align*}
Since $\psi$ is non-increasing, the function $t \mapsto \psi(t/R^{2s/\alpha})$ is non-increasing, and therefore $\cd_{t|T}^{\alpha}\psi_R(t, x) \geq 0$ for $T$ sufficiently large. As $u_0(x) \geq 0$, the integrand is nonnegative, so
$$\int\limits_0^T \int\limits_{\mathbb{R}^N} u_0(x) \, \cd_{t|T}^{\alpha} \psi_R(t,x) \, dx \, dt \geq 0,$$
together with the previously established estimates, implies that
\begin{eqnarray*}
\int\limits_0^{TR^{\frac{2\ts}{\alpha}}} t^{\sigma} \psi\left(\frac{t}{R^{\frac{2\ts}{\alpha}}}\right)\, dt \int\limits_{\mathbb{R}^N} \bw(x) \psi\left(\frac{|x|}{R}\right)\, dx \leqslant C R^{N+\frac{2\ts}{\alpha}-\frac{2\ts p}{p-1}}
\end{eqnarray*}
On the other hand, we take into account that $\bw$ is in $L^1$. Then, for $R$ sufficiently large, we obtain by the Dominated Convergence Theorem that
$\int\limits_{\mathbb{R}^N} \bw(x) \psi\left(\frac{|x|}{R}\right) d x \longrightarrow \int\limits_{\mathbb{R}^N} \bw(x) d x\geq 0$
Moreover, we can bound the second term on the right-hand side of \eqref{EQ} from below as follows
\begin{align*}
\int\limits_0^{TR^{\frac{2\ts}{\alpha}}}\int\limits_{\mathbb{R}^N} \bw(x) \psi_R (t,x)dxdt &\geqslant \int\limits_{\frac{TR^{\frac{2\ts}{\alpha}}}{2}}^{TR^{\frac{2\ts}{\alpha}}} t^{\sigma} \psi\left( \frac{t}{R^{\frac{2\ts}{\alpha}}}\right) d t \int\limits_{\mathbb{R}^N} \psi\left( \frac{|x|} {R}\right) \bw(x) \, dx \\
&\geq R^{\frac{2\ts}{\alpha}(\sigma+1)}\int\limits_{\frac{1}{2}}^1 \tau^{\sigma} \psi\left( \tau\right) d \tau \int\limits_{\mathbb{R}^N} \psi\left( \frac{|x|} {R}\right) \bw(x)\,dx.
\end{align*}
More precisely, we obtain
\begin{align*}
\int_{\mathbb{R}^N}  \bw(x)\,dx \leq C R^{\gamma}, \text{where} \quad \gamma = N - \frac{2\ts p}{p-1} - \frac{2\ts\sigma}{\alpha}.
\end{align*}
Observing that  $1 < p < \frac{N\alpha-2\ts\sigma}{N\alpha-2\ts(\alpha+\sigma)}$ and letting $R\to \infty$, we obtain $\int_{\mathbb{R}^N} \bw(x)  dx \leq 0,$
which is clearly a contradiction. This completes the proof.
\subsection{Proof of Theorem~\ref{thm:glob-exis-1}}
\label{subsec:glob-1}
The global existence result of the theorem in the case without a forcing term has been studied in \cite{KLT}. In that setting, the authors establish their result by employing a contradiction argument combined with suitable test functions and scaling techniques, through which they identify the critical exponents governing global existence.

In this section, we extend the analysis to include the presence of a forcing term. The critical exponent $p_{c}$ defined in \eqref{p-c}, together with the exponent $r_{c}$ appearing in Theorem~\ref{thm:glob-exis-1}, plays a fundamental role in the development of the global existence analysis presented here.

We can select a positive constant $q$ that satisfies:
\begin{equation}
\label{Condition-q}  
\max\left(\frac{2\ts}{Np(p-1)},\frac{2\ts\alpha+2s\sigma(p-1)}{N\alpha(p-1)}\right)<\frac{1}{q}<\min \left(\frac{2\ts}{N(p-1)},\frac{N \alpha-2\ts(\alpha+\sigma)}{N\alpha-2\ts\sigma}\right), \quad q\geq p
\end{equation}
Moreover, the assumptions $\alpha \in (0,1)$, $p>p^*= \frac{N\alpha-2\ts\sigma}{N\alpha-2\ts(\alpha+\sigma)}$ and $-\alpha < \sigma < 0$ imply that $p_c, {r_c} >1$. We now establish a sufficient condition for the existence of a global solution. Specifically, we prove that there exists a constant $\epsilon > 0$, depending solely on $p, q, \alpha, \sigma$, and $N$, such that for any initial data $u_0 \in L^{p_c}$ and $\bw \in L^{r_c}$ satisfying the smallness condition $\|u_0\|_{L^{p_c}} + \|\bw\|_{L^{r_c}} < \epsilon$, a global solution exists. Here, $q$ is considered a fixed parameter, and we define
\begin{equation*}  
\beta-\frac{N \alpha}{2\ts}\left(\frac{1}{p_c}-\frac{1}{q}\right)=\alpha-\beta(p-1)-\frac{\alpha N}{2\ts}\left(\frac{p-1}{q}\right)=\beta+\alpha-\frac{N \alpha}{2\ts}\left(\frac{1}{{r_c}}-\frac{1}{q}\right)+\sigma=0 
\end{equation*}
\begin{equation*}
\beta= \frac{\alpha}{p-1}-\frac{\alpha N}{2q\ts}   \end{equation*}
It is straightforward to verify from condition \eqref{Condition-q} that $0 < \beta p \leq \alpha < 1$, where the last inequality is strict when $\alpha \in (0,1)$. Since $q \geq p$, this further implies $0 < \beta \leq \frac{1}{q}$, which is the required bound. 
Now, let $\delta$ be a sufficiently small positive constant. We define the function space $Y$ as follows
$$
Y = \left\{ u \in  L_{Loc}^\infty((0,\infty),L^q(\mathbb{R}^N)) \; ; \: 
 \sup_{t>0} t^{\beta}\|u(t)\|_{L^q} \leqslant \delta \right\}$$
with the distance defined as $d(u, v)=\sup_{t>0} t^{\beta}\|u-v\|_{L^q} $, it is evident that $(Y,d)$ forms a complete metric space. We will demonstrate that the equation
\begin{equation*}
\begin{aligned}
\mathcal{G}u(t,x) = & \int\limits_{\mathbb{R}^N} u_0(y) Z(t,x-y)  dy + \int\limits_0^t \int\limits_{\mathbb{R}^N} |u(y,\tau)|^{p} Y(t-\tau,x-y)  dy  d\tau\\& + \int\limits_0^t \int\limits_{\mathbb{R}^N} \tau^{\sigma}\bw(y) Y(t-\tau,x-y)  dy  d\tau.
\end{aligned}
\end{equation*}
has a unique fixed point in $Y$.
Assuming $u_0 \in L^{p_c}$ and using Lemma \ref{lem:kernels}, we select $r$ satisfying  
$1 - \frac{1}{r} = \frac{1}{p_c} - \frac{1}{q}$. From condition~\eqref{Condition-q}, we observe that $1 - \frac{1}{r} \in \big(0, \frac{2\ts}{N}\big)$, which implies
\begin{align*}
t^{\beta} \|\mathcal{L}_{u_0}(t,\cdot) \|_{L^q(\mathbb{R}^N)} &\leq \| u_0 \|_{L^{p_c}(\mathbb{R}^N)} \| Z(t, \cdot) \|_{L^r(\mathbb{R}^N)}\\
&\leq t^{\beta-\frac{N\alpha}{2\ts}\left(1-\frac{1}{r}\right)}\|u_0\|_{L^{p_c}}\|F\|_{L^r}:=\eta \|F\|_{L^{r}}.
\end{align*}
In the same manner, we choose $m$ such that
$\frac{1}{m} = 1 - \frac{p}{q} + \frac{1}{q}.$
Since $p \geq 1$ and $q>\frac{N\alpha}{N\alpha-2\ts(\alpha+\sigma)}$, we have $0 \leq 1 - \frac{1}{m} < \frac{2\ts}{N}$, or equivalently, $m \in [1, \frac{N}{N-2s\ts})$. 
Hence, \eqref{lem:kernels} applies and yields the estimate $1 - \alpha + \frac{N\alpha}{2\ts}\left(1 - \frac{1}{m}\right) < 1.$
Thus, for $0 \leq t \leq t_0$,
\begin{align*}
\left\| \mathcal{N}(u)(t, \cdot) \right\|_{L^q(\mathbb{R}^N)} 
&= \int\limits_0^t \left\| |u(s, \cdot)|^p \right\|_{L^{\frac{q}{p}}(\mathbb{R}^N)} \left\| Y(t-\tau,\cdot) \right\|_{L^m(\mathbb{R}^N)}  d\tau \\
&\leq \left\| G \right\|_{L^m(\mathbb{R}^N)} \sup_{\tau>0} \tau^{\beta p}\left\| u(\tau, \cdot) \right\|_{L^q(\mathbb{R}^N)}^p \int\limits_0^t \tau^{-\beta p}(t-\tau)^{-(1-\alpha+\frac{N\alpha}{2\ts q}(p-1))}  d\tau \\
&=\delta^p t^{\alpha-\frac{N\alpha}{2\ts q}(p-1)-\beta p}\int\limits_0^1 \tau^{-\beta p}(1-\tau)^{-(1-\alpha+\frac{N\alpha}{2\ts q}(p-1))} d\tau \left\| G \right\|_{L^m(\mathbb{R}^N)}\\
&=\delta^p t^{-\beta} \mathscr{B}\left(1-\beta p, \alpha - \frac{\alpha N}{2\ts}\left(\frac{p-1}{q}\right)\right)\left\| G \right\|_{L^m(\mathbb{R}^N)} 
\end{align*}
Here, $\mathscr{B}$ represents the beta function. We observe that by \eqref{Condition-q} and the fact that $\beta p \leqslant \beta q< 1$, $\mathscr{B} $ is well-defined.

Proceeding similarly, we define $1+\frac{1}{q} = \frac{1}{{r_c}} + \frac{1}{\varrho}$. 
The condition on $q$ given in \eqref{Condition-q} ensures that $1 - \frac{1}{\varrho} \in (0, \frac{2\ts}{N})$. 
Given that $\sigma > -1$, we may use lemma \eqref{lem:kernels} to obtain
\begin{align*}
\|\mathcal{S}_\bw(t, \cdot)\|_{L^q(\mathbb{R}^N)} &\leq \|\bw\|_{L^{{r_c}}(\mathbb{R}^N)} \|G\|_{L^\varrho} \int\limits_0^t \tau^{\sigma} (t-\tau)^{-(1-\alpha+\frac{N\alpha}{2\ts}(1-\frac{1}{\varrho}))} d\tau \\
&\leq t^{\alpha-\frac{N\alpha}{2\ts}(1-\frac{1}{\varrho})+\sigma} \|\bw\|_{L^{{r_c}}(\mathbb{R}^N)} \|G\|_{L^\varrho} \int\limits_0^1 \tau^{\sigma} (1-\tau)^{\alpha-1-\frac{N\alpha}{2\ts}(1-\frac{1}{\varrho})} d\tau \\
&= t^{-\beta} \mathscr{B}\left(\sigma+1, \alpha-\frac{N\alpha}{2\ts}(1-\frac{1}{\varrho})\right) \|\bw\|_{L^{{r_c}}(\mathbb{R}^N)} \|G\|_{L^\varrho(\mathbb{R}^N)}.
\end{align*}
Therefore, based on the estimates above, choosing sufficiently small $\epsilon$ and $\delta>0$ leads to the conclusion that
\begin{equation*}
\sup_{t>0} t^\beta \| (\mathcal{G} u)(t) \|_{L^r} \leq C \left( \| u_0 \|_{L^{p_c}} + \delta^{p} + \| \bw \|_{L^{{r_c}}} \right) \leqslant C \epsilon
\end{equation*}
This shows that $\mathcal{G}$ is a mapping from $Y$ into itself. It remains to demonstrate that the self-mapping $\mathcal{G}: Y \rightarrow Y$ is a contraction to complete the proof.
By applying similar arguments as mentioned above, we can establish that for functions $u$ and $v$ in $L_{Loc}^{\infty}((0,\infty), L^q(\mathbb{R}^N))$
\begin{align*}
t^\beta \|\mathcal{G}(u) - \mathcal{G}(v)\|_{L^q(\mathbb{R}^N)} &\leq Ct^\beta \|G\|_{L^{r}} \int\limits_0^t (t - \tau)^{\alpha-1-\frac{\alpha N}{2\ts}\left(\frac{p-1}{q}\right)} \|u^p - v^p\|_{L^{\frac{q}{p}}(\mathbb{R}^N)} \, d\tau \\
&\leq C \|G\|_{L^{r}} t^\beta \int\limits_0^t (t - \tau)^{\alpha-1-\frac{\alpha N}{2\ts}\left(\frac{p-1}{q}\right)} \left(\|u\|^{p-1}_{L^q}+\|v\|^{p-1}_{L^q}\right) \|u - v\|_{L^q(\mathbb{R}^N)} \, d\tau \\
&\leq Ct^\beta \|G\|_{L^{r}} \delta^{p-1} \int\limits_0^t (t - \tau)^{\alpha-1-\frac{\alpha N}{2\ts}\left(\frac{p-1}{q}\right)} \tau^{-p\beta} \, d\tau \|u - v\|_Y \\
&= C\delta^{p-1} t^{\beta-p\beta-\alpha-\frac{\alpha N}{2\ts}\left(\frac{p-1}{q}\right)} \int\limits_0^1 (1 - \tau)^{\alpha-1-\frac{\alpha N}{2\ts}\left(\frac{p-1}{q}\right)} \tau^{-p\beta} \, d\tau  \|u - v\|_Y \\
&= C\delta^{p-1} \mathscr{B}\left(1-p\beta,\alpha-\frac{\alpha N}{2\ts}\left(\frac{p-1}{q}\right)\right)\|u - v\|_Y.
\end{align*}
Hence, $\mathcal{G}$ is clearly a contraction on $Y$. Thus, by the contraction mapping principle, $\mathcal{G}$ has a unique fixed point $u$ in $Y$ which is the desired  solution.
\subsection{Proof of Theorem~\ref{thm:glob-exis-2}}
\label{subsec:glob-2}
The proof consists of verifying that the following inequality from Lemma~\ref{lem:W-Condition} holds:
$$t^\beta \left( \|\mathcal{L}_{u_0}(t, \cdot)\|_{L^q} + \|\mathcal{S}_\bw(t, \cdot)\|_{L^q} \right) \leq \mathscr{M} \quad \text{for all } t > 0,$$
where $\beta = \frac{\alpha}{p-1} - \frac{N\alpha}{2\ts q}.
$ We proceed considering cases $ t > R_0$ and $t \leq R_0$ separately.
For $ t \leq R_0 $, we exploit the contraction properties of the kernels $ Z $  and $ Y $ given in Lemma \ref{lem:kernels} to obtain $ \|\mathcal{L}_{u_0}(t,\cdot)\|_{L^q} \leq \|u_0\|_{L^q} $. For the source term, Young's inequality and the bound $ \|Y(t,\cdot)\|_{L^1} \leq C t^{\alpha-1} $ yield
$$
\|\mathcal{S}_\bw(t,\cdot)\|_{L^q} \leq \|\bw\|_{L^q} \int\limits_0^t \tau^\sigma \|Y(t-\tau,\cdot)\|_{L^1}  d\tau \leq C \|\bw\|_{L^q} t^{\alpha+\sigma}.
$$
Choosing
$$
R_0 = \min\left\{ \left( \frac{\mathscr{M}}{2\|u_0\|_{L^q}} \right)^{1/\beta}, \left( \frac{\mathscr{M}}{2C\|\bw\|_{L^q}} \right)^{1/(\beta+\alpha+\sigma)} \right\},
$$
we find that
$$
t^\beta \|\mathcal{L}_{u_0}(t,\cdot)\|_{L^q} \leq R_0^\beta \|u_0\|_{L^q} \leq \frac{\mathscr{M}}{2}, \quad \text{and} \quad t^\beta \|\mathcal{S}_\bw(t,\cdot)\|_{L^q} \leq C R_0^{\beta+\alpha+\sigma} \|\bw\|_{L^q} \leq \frac{\mathscr{M}}{2}.
$$
This establishes the desired bound on the interval $ [0, R_0] $.
For $t > R_0$, we analyze the behavior by decomposing the spatial domain into three regions:
\begin{align*}
\mathcal{E}_1 &= \{ y \in \mathbb{R}^N : |y| > M t^{\frac{\alpha}{2\ts}} \}, \\
\mathcal{E}_2 &= \{ y \in \mathbb{R}^N : |y| < M t^{\frac{\alpha}{2\ts}},\ |x-y| < \delta t^{\frac{\alpha}{2\ts}} \}, \\
\mathcal{E}_3 &= \{ y \in \mathbb{R}^N : |y| < M t^{\frac{\alpha}{2\ts}},\ |x-y| > \delta t^{\frac{\alpha}{2\ts}} \}.
\end{align*}
We decompose the linear part as $\mathcal{L}_{u_0} = \mathcal{L}_{u_0,1} + \mathcal{L}_{u_0,2} + \mathcal{L}_{u_0,3}$, where each term corresponds to integration over $\mathcal{E}_1$, $\mathcal{E}_2$, $\mathcal{E}_3$ respectively. Similarly, we write $\mathcal{S}_{\mathbf{\bw}} = \mathcal{S}_{\mathbf{\bw,1}} + \mathcal{S}_{\mathbf{\bw,2}} + \mathcal{S}_{\mathbf{\bw,3}}$, with
$$\mathcal{S}_{\mathbf{\bw,j}}(t,x) = \int\limits_0^t \tau^\sigma \int\limits_{\mathcal{E}_j} \mathbf{\bw}(y) Y(t-\tau, x-y)  dy d\tau, \quad j={1,2,3}.$$
We now estimate each term separately. For $\mathcal{L}_{u_0,1}$, using \eqref{H1} and the kernel estimate $\|Z(t,\cdot)\|_{L^r} = \|F\|_{L^r} t^{-\frac{N\alpha}{2\ts}(1-\frac{1}{r})}$, we obtain
$$t^\beta \|\mathcal{L}_{u_0,1}(t,\cdot)\|_{L^q} \leq t^\beta \|u_0\|_{L^{p_c}(\mathcal{E}_1)} \|Z(t,\cdot)\|_{L^r} \le \frac{\mathscr{M}}{6}.$$
For $\mathcal{S}_{\bw,1}$, condition \eqref{H2} and the bound $\|Y(t,\cdot)\|_{L^r} \le C t^{1-\alpha-\frac{N \alpha}{2\ts}(1-\frac{1}{r})}$ give
$$t^\beta \|\mathcal{S}_{\bw,1}(t,\cdot)\|_{L^q} \le t^\beta \|\bw\|_{L^{p_c}(\mathcal{E}_1)} \int\limits_0^t \tau^\sigma \|Y(t-\tau,\cdot)\|_{L^r} d\tau \leq \frac{\mathscr{M}}{6}.$$
The local terms $\mathcal{L}_{u_0,2}$ and $\mathcal{S}_{\mathbf{\bw,2}}$ are controlled via \eqref{H3} and the local boundedness of $F$ and $G$:
$$t^\beta \|\mathcal{L}_{u_0,2}(t,\cdot)\|_{L^q} \leq \frac{\mathscr{M}}{6}, \quad t^\beta \|\mathcal{S}_{\mathbf{\bw,2}}(\cdot,t)\|_{L^q} \leq \frac{\mathscr{M}}{6}.$$
Finally, the nonlocal terms $\mathcal{L}_{u_0,3}$ and $\mathcal{S}_{\mathbf{\bw,3}}$ are estimated using \eqref{A1} and \eqref{A2}, which exploit the tail behavior of $F$ and the averaged smallness of $u_0$ and $\mathbf{\bw}$:
$$t^\beta \|\mathcal{L}_{u_0,3}(t,\cdot)\|_{L^q} \leq \frac{\mathscr{M}}{6}, \quad t^\beta \|\mathcal{S}_{\mathbf{\bw,3}}(t,\cdot)\|_{L^q} \leq \frac{\mathscr{M}}{6}.$$
Summing these  bounds yields 
\begin{equation*}
t^\beta \left(\|\mathcal{L}_{u_0}\|_{L^q} + \|\mathcal{S}_{\bw}\|_{L^q} \right) \leq \mathscr{M} \quad \text{for } t > R_0.
\end{equation*}
The global existence of the solution follows from Lemma \ref{lem:W-Condition}, which concludes the proof.
\section{Numerical simulation}
\label{sec:num}

In this section, we present numerical simulations illustrating the qualitative
behavior predicted by Theorem~\ref{thm:blowup}.
We focus on the one-dimensional fractional diffusion equation
\begin{equation}\label{eq:num}
\partial_t^{1/2} u(t,x) + (-\Delta)^{1/4} u(t,x)
= |u(t,x)|^p + e^{-|x|^2},
\qquad (t,x)\in (0,\infty)\times\mathbb{R},
\end{equation}
which corresponds to the general model studied in this paper with
$\alpha=\tfrac12$, $s=\tfrac14$ and $\sigma=0$.
The forcing term $e^{-|x|^2}$ is nonnegative, localized in space, and satisfies
$\int_{\mathbb{R}} e^{-|x|^2}\,dx>0$, so that all assumptions of
Theorem~\ref{thm:blowup} are fulfilled.

In this setting, the critical exponent defined in
Theorem~\ref{thm:blowup} reads
\[
p^*=\frac{N\alpha-2s\sigma}{N\alpha-2s(\alpha+\sigma)}=2,
\]
and therefore the theorem predicts the nonexistence of global weak solutions
for $1<p<p^*$.

We numerically compare a subcritical case $p=1.5<p^*$ with a supercritical case
$p=2.5>p^*$ in order to visualize the sharpness of this threshold.
\FloatBarrier

\begin{figure}[!ht]
  \centering
  \includegraphics[width=0.8\textwidth]{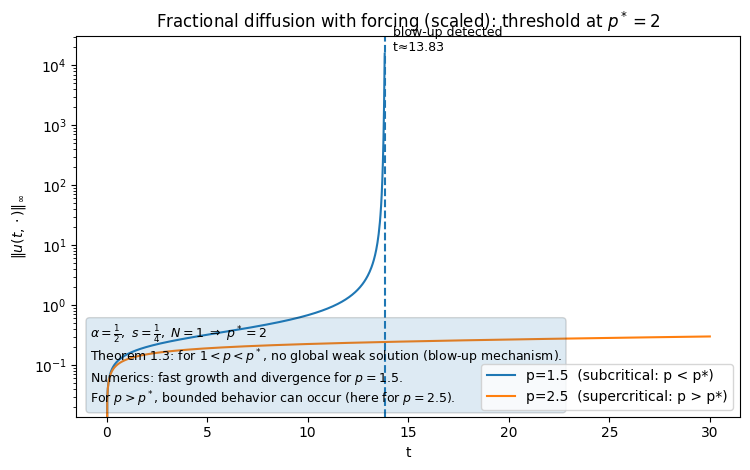}
  \caption{Time evolution of $\|u(t,\cdot)\|_{L^\infty(\mathbb{R})}$ for
  $\alpha=\tfrac12$ and $s=\tfrac14$.}
  \label{fig:blowup-time}
\end{figure}

\begin{figure}[!ht]
  \centering
  \includegraphics[width=0.8\textwidth]{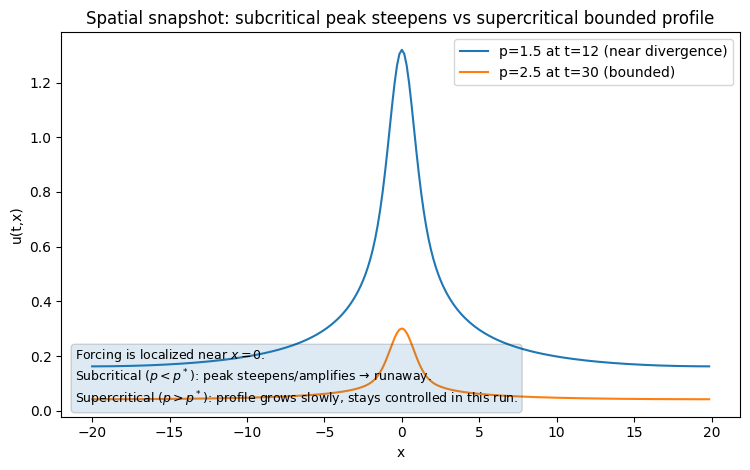}
  \caption{Spatial profiles $u(t,x)$ at representative late times for
  $\alpha=\tfrac12$ and $s=\tfrac14$.}
  \label{fig:blowup-space}
\end{figure}

\FloatBarrier

Figure~\ref{fig:blowup-time} displays the time evolution of the supremum norm
$\|u(t,\cdot)\|_{L^\infty(\mathbb{R})}$.
For $p=1.5<p^*$, the numerical solution grows rapidly and becomes unbounded
in finite time, which is consistent with the blow-up (or nonexistence)
mechanism established in Theorem~\ref{thm:blowup}.
In contrast, for $p=2.5>p^*$, the solution remains bounded throughout the
simulation, indicating a fundamentally different qualitative behavior.

Figure~\ref{fig:blowup-space} complements this observation by highlighting the
spatial mechanism associated with the blow-up.
In the subcritical regime, the forcing induces a strong concentration near
$x=0$, whose amplitude increases rapidly and leads to loss of boundedness.
For $p>p^*$, the solution exhibits a much milder growth and no such
concentration is observed within the simulated time range.

We stress that these numerical results do not constitute a proof of either
blow-up or global existence.
Rather, they provide a qualitative visualization of the critical threshold
identified in Theorem~\ref{thm:blowup}.
In particular, while the theorem guarantees nonexistence of global weak
solutions for $1<p<p^*$ under the stated assumptions, it does not exclude
blow-up for $p>p^*$ when the forcing is sufficiently large, nor does it assert
global existence without additional conditions, such as the smallness
assumptions appearing in Theorem~\ref{thm:glob-exis-1}.

\vspace{1cm}

\hrule 

\vspace{0.3cm}
\noindent{\bf\large Declarations.} {\em On behalf of all authors, the corresponding author states that there is no conflict of interest. No data-sets were generated or analyzed during the current study.}
%%%%%%%%%%%%%%%%%%%%%%%%%%%%%%%%%%%%%%%%%%%%%%%
	\vspace{0.3cm}
 \hrule

\end{document}